\documentclass[a4paper]{article}
\usepackage{graphicx}
\usepackage{amssymb}
\usepackage{stmaryrd}
\usepackage{amsthm}
\usepackage{amssymb}
\usepackage{comment}
\usepackage{bbm}
\usepackage[english]{babel}
\newtheorem{lemma}{Lemma} 
\newtheorem{corollary}{Corollary}
\newtheorem{theorem}{Theorem}

\newtheorem{proposition}{Proposition}

\theoremstyle{remark}
\newtheorem{remark}{Remark}

\theoremstyle{definition}
\newtheorem{definition}{Definition}
\newtheorem{algorithm}{Algorithm}

\newcommand{\outv}{{\scriptscriptstyle{out}}}
\newcommand{\inv}{{\scriptscriptstyle{in}}}

\newcommand{\Tplanar}{{T_\circ}}
\newcommand{\Tplanarc}{{T_c}}
\newcommand{\zc}{{z_{m,D}^{(c)}}}
\newcommand{\tc}{{t_{m,D}^{(c)}}}
\newcommand{\setcells}{{\mathcal{F}_{m,D}^{\circ\circ}}}

\author{Guillaume Chapuy, LIX, \'Ecole Polytechnique}
\title{Asymptotic enumeration of constellations and related families of maps
 on orientable surfaces.} 

\begin{document}
 
\maketitle

\begin{abstract}
We perform the asymptotic enumeration of two classes of rooted maps on
orientable surfaces of genus $g$: $m$-hypermaps and $m$-constellations. 
%These maps generalise 
For $m=2$ they correspond 
respectively to maps with even face degrees and bipartite
maps. We obtain explicit asymptotic formulas for the number of such maps with any
finite set of allowed face degrees.

Our proofs rely on the generalisation to orientable surfaces of the
Bouttier-Di Francesco-Guitter bijection, which reduces maps to certain
arrangements of labelled multitype trees, and on the study the corresponding
generating series via algebraic methods. We also show that each of the $2g$ fondamental 
cycles of the surface contributes a factor $m$ between the numbers of
$m$-hypermaps and $m$-constellations --- for example, 
large maps of genus $g$ with
even face degrees are bipartite with probability tending to $\frac{1}{2^{2g}}$. 

A special case of our results implies former conjectures of Gao.
\end{abstract}

\section{Introduction}
\label{sec:in}

Maps are combinatorial objects which describe the embedding of a graph in a surface.
The enumeration of maps began in the sixties with the works of Tutte, in the
series of papers
\cite{Tutte:census-triangulations,Tutte:census-slicings,Tutte:census-Hamiltonian,Tutte:census-maps}.
By analytic techniques, involving recursive decompositions and non trivial manipulations of power series, Tutte
%studied a certain number of families of planar maps, and 
obtained beautiful and simple enumerative formulas for several families of planar maps.
His techniques were extended in the late eighties by several authors to more
sophisticated families of maps or to the case of maps of higher genus. 
Bender and Canfield (\cite{BeCa0,BeCa}) obtained the asymptotic number 
of maps on a given orientable surface. Gao (\cite{Gao}) obtained formulas for
the asymptotic number of $2k$-angulations on orientable surfaces, and
conjectured a formula for more general families (namely maps where the degrees of the faces are restricted to lie in
a given finite subset of $2\mathbb{N}$).

A few years later, Schaeffer (\cite{Sc:PhD}), following the work of Cori and Vauquelin (\cite{CoVa}), 
gave in in thesis a bijection between planar maps and certain labelled trees
which enables to recover the formulas of Tutte, and explains combinatorially 
their remarquable simplicity. This bijection has suscited a lot of interest in probability and physics, 
since it also enables to study geometrical aspects of large random maps
(\cite{ChassaingSc,LG,LGPa,BG,Miermont:Geodesiques,LG:Geodesiques}). Moreover, it has been
generalised in two directions. First, Bouttier, Di Francesco, and Guitter (\cite{BDFG}) gave
a construction that generalises Schaeffer's bijection to the large class of \emph{Eulerian maps}, which
includes for example maps with restricted face degrees, or constellations.
Secondly, Marcus and Schaeffer (\cite{MaSc}) generalised Schaeffer's construction to the case of maps
drawn on orientable surfaces of any genus, opening the way to a bijective
derivation (\cite{ChMaSc}) of the results of Bender and Canfield.

The first purpose of this article is to unify the two generalisations of Schaeffer's bijection: we show 
that the general construction of Bouttier, Di Francesco, and Guitter
stays valid in any genus, and involves the same kind of objects as developped in \cite{MaSc}.
Our second (and main) task is then to use this bijection to perform the
asymptotic enumeration of several families of maps, namely $m$-constellations
and $m$-hypermaps.  
These maps will be defined later, but let us mention now
that for $m=2$, they correspond respectively to bipartite maps, and maps with even
face degrees. In particular, a special case of our results implies the
conjectures of Gao(\cite{Gao}).
The proofs rely on a decomposition of the objects inherited
from the bijection to a finite number of combinatorial objects, which we call
the \emph{full schemes}, from which all the objects can be reconstructed using
an arrangement of lattice paths. The analysis of these paths and
their arrangements is then performed via generating series techniques.
The link between hypermaps and constellations
relies on elementary algebraic graph theory: we introduce a
discrete vector space, the \emph{typing space}, that encodes the length modulo
$m$ of the $2g$ fundamental cycles of $m$-hypermaps, and we show that large
hypermaps project asymptotically uniformly on that space. For example, for $m=2$, this
shows that maps of genus $g$ with even face degrees are bipartite with
probability tending to $\left(\frac{1}{2}\right)^{2g}$ when their size goes to infinity.

\section{Definitions and main results}
\label{sec:defandresults}

Let $\mathcal{S}_g$ be the torus with $g$-handles. A \emph{map} on $\mathcal{S}_g$
(or map of genus $g$) is a proper embedding of a finite graph $G$ in $\mathcal{S}_g$ 
such that the maximal connected components of $\mathcal{S}_g \setminus G$ are 
simply connected regions. Multiple edges an loops are allowed.
The maximal simply connected components are called the \emph{faces} of the map.
The degree of a face is the number of edges incident to it, counted with multiplicity.
A \emph{corner} consists in a vertex together with an
angular sector adjacent to it.

We consider maps up to homeomorphism, i.e. we identify two maps such that there exists
an orientation preserving homeorphism that sends one to another. In this
setting, maps become purely combinatorial objects (see \cite{Mohar-Thomassen} for
a detailed discussion on this fact). In particular, there are only a finite
number of maps with a given number of edges, opening the way to enumeration
problems.

All the families of maps considered in this article will eventually be 
rooted (which means that an edge has been distinguished and oriented), pointed
(when only a vertex has been distinguished), or both rooted and pointed.
In every case, the notion of oriented homeomorphism involved in the definition
of a map is adapted in order to keep trace of the pointed vertex or edge. 

The first very usefull result when working with maps on surfaces is \emph{Euler
characteristic formula}, that says that if a map of genus $g$ has
$f$ faces, $v$ vertices, and $n$ edges, then we have:
$$
v+f = a +2 -2g.
$$

An \emph{Eulerian map} on $\mathcal{S}_g$ is a map on $\mathcal{S}_g$, together with
a colouring of its faces in black and white, such that only faces of different colours are
adjacent. By convention, the root of an Eulerian map will always be oriented with a black face
on its right. This article will mainly be concerned with two special cases of
Eulerian maps, namely $m$-hypermaps and $m$-constellations.
\begin{definition}
\label{def:constellation}
Let $m\geq 2$ be an integer. An \emph{$m$-constellation} on $\mathcal{S}_g$ 
is a map on $\mathcal{S}_g$, together  with a colouring of its faces in black and white
such that:
\begin{itemize}
\item[{\bf(i)}]
only faces of different colours are adjacent.
\item[{\bf(ii)}]
black faces have degree $m$, and white faces have a degree which is a 
multiple of $m$.
\item[{\bf(iii)}]
every vertex can be given a label in $\{1,\dots,m\}$ such that around every black face, 
the labels of the vertices read in clockwise order are exaclty $1,\dots,m$.
\end{itemize}
A map that satisfies conditions {\bf (i)} and {\bf (ii)} is called a
$m$-hypermap.
\end{definition}
It is a classical fact that in the planar case, conditions $(i)$ and $(ii)$ imply
condition $(iii)$, so that all planar $m$-hypermaps are in fact
$m$-constellations ; however, this is not the case in higher genus.
% A map which
%satisfies conditions $(i)$ and $(ii)$ will be called here a
%\emph{$m$-hypermap}. The enumeration of these maps will be adressed in the
%last section (\ref{sec:mhypermaps}).
Observe that $2$-hypermaps are in bijection with maps whose all faces have even
degree (in short, \emph{even maps}). This correspondance relies on contracting
every black face of the $2$-hypermap to an edge of the even map. Observe
also that this correspondance specializes to a bijection between
$2$-constellations and bipartite maps. This makes $m$-constellations and
$m$-hypermaps a natural object of study. For previous enumerative studies on
constellations, and for their connection to the theory of enumeration of rational functions on a surface, 
see \cite{LaZv},\cite{MBM-Sc}.

In the rest of the paper, $m\geq 2$ will be a fixed integer, and
$D\subset\mathbb{N}_{>0}$  will be a non-empty and finite subset of the
positive integers. If $m=2$, we assume furthermore that $D$ is not reduced to
$\{1\}$. A \emph{$m$-hypermap} with degree set $mD$ is a $m$-hypermap
in which all white faces have a degree which belongs to $mD$. The same
definition holds for constellations. For example, a 
$2$-constellation of degree set $2\{2\}$ is (up to contracting black faces to
edges) a bipartite quadrangulation. Finally, the \emph{size} of a $m$-hypermap is its number
of black faces.

Our main results are the two following theorems:
\begin{theorem}%[\ref{thm:main}]
\label{thm:main}
The number $c_{g,D,m}(n)$ of rooted $m$-constellations of genus $g$, 
degree set $mD$, and size $n$ satisfies: 
\begin{eqnarray*}
c_{g,D,m}(n) \sim 
  t_g \frac{\gcd(D)}{2}  
             \left(\frac{(m-1)^{5/2}\sqrt{2\gamma_{m,D}}} {
             m\beta_{m,D}^{5/2}} \right)^{g-1} 
            n^{\frac{5(g-1)}{2}}
             {({z_{m,D}^{(c)}})}^{-n}
\end{eqnarray*}
when $n$ tends to infinity
along multiples of $\gcd(D)$, and where:
\begin{itemize}
  \item[$\bullet$] $t_{m,D}^{(c)}$ is the smallest positive root of: 
  $\displaystyle\sum_{k\in D}[(m-1)k-1] {mk-1 \choose k} [t_{m,D}^{(c)}]^{k} =1 $
  \item[$\bullet$] $\displaystyle \beta_{m,D} = \sum_{k\in D} [(m-1)k] {mk-1
  \choose k} [t_{m,D}^{(c)}]^{k} $
  \item[$\bullet$] $\displaystyle \gamma_{m,D} = \sum_{k\in D}
  [(m-1)k][(m-1)k-1] {mk-1 \choose k} [t_{m,D}^{(c)}]^{k} $
  \item[$\bullet$]
  $\displaystyle z_{m,D}^{(c)}=t_{m,D}^{(c)}[\beta_{m,D}]^{1-m}$
\end{itemize}
and the constant $t_g$ is defined in \cite{BeCa0}.
\end{theorem}

\begin{theorem}%[\ref{thm:hypermaps}]
\label{thm:hypermaps}
The number $h_{g,D,m}(n)$ of rooted $m$-hypermaps of degree set $mD$
and size $n$ on a surface of genus $g$ satisfies: 
\begin{eqnarray*}
h_{g,D,m}(n) &\sim& m^{2g} {c}_{g,M,D} (n)
%\\
%&\sim&
%  t_g \frac{m^2(m-1)\gcd(D)}{2}  
%             \left(\frac{m(m-1)^{5/2}\sqrt{2\gamma_{m,D}}} {
%             \beta_{m,D}^{5/2}} \right)^{g-1} 
%            n^{\frac{5(g-1)}{2}}
%             {({z_{m,D}^{(c)}})}^{-n}
\end{eqnarray*}
when $n$ tends to infinity
along multiples of $\gcd(D)$.
\end{theorem}

Observe that Theorem~\ref{thm:hypermaps} can be reformulated as follows: the
probability that a large $m$-hypermap of genus $g$ is a $m$-constellation tends
to $1/m^{2g}$. To our knowledge, this fact had only been observed in the case
of quadrangulations (which are known to be bipartite with probability
$\sim 1/4^g$, see \cite{Be:overview} and references therein). Putting
Theorems~\ref{thm:main} and~\ref{thm:hypermaps} together gives an asymptotic
formula for the number $h_{g,D,m}(n)$, which was already proved by Gao in the
case where $m=2$ and $D$ is a singleton, and conjectured for $m=2$ and
general $D$ in the paper~\cite{Gao}. All the other cases were, as far as we
now, unknown.

\section{The Bouttier, Di Francesco, and Guitter's bijection on an orientable
surface.}
\label{sec:bij}

In this section, we describe the Bouttier, Di Francesco, and Guitter's
bijection on $\mathcal{S}_g$. This construction
has been introduced in \cite{BDFG}, as a generalisation of the Cori-Vauquelin-Schaeffer bijection,
and  provides a correspondance between planar maps and plane trees. Here, we unify
this generalisation with the generalisation of \cite{MaSc}, where the classical bijection
of Cori, Vauquelin, and Schaeffer is extended to any genus.

All the constructions are local and are similar to the planar case. The proof
that the construction is well defined (Lemma~\ref{lemma:welldef})is an
adaptation of the one of \cite{MaSc}, and is different of the one of
\cite{BDFG}, that uses the planarity. After that, everything
(Lemma~\ref{lemma:BDFG}) is already contained in \cite{BDFG}. 
In particular, we will not state all proofs.

\subsection{From maps to mobiles.}
\begin{figure}[h]
\centerline{\includegraphics[scale=1]{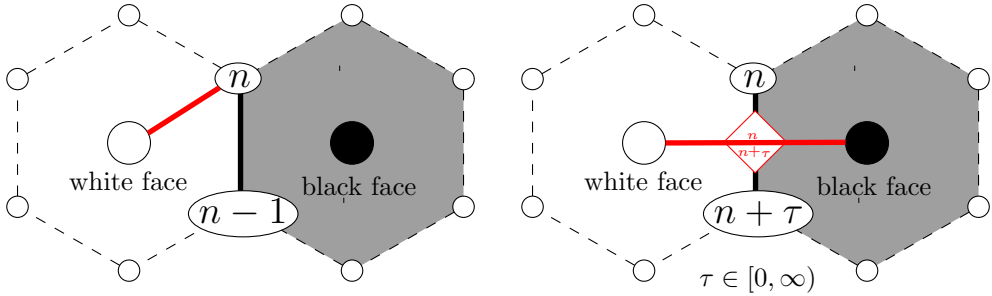}}
\caption{The Bouttier-Di Francesco-Guitter construction.}
\label{fig:BDFGrules}
\end{figure}
Let $\mathfrak{m}$ be a rooted and pointed Eulerian map on
$\mathcal{S}_g$ (i.e. $\mathfrak{m}$ has at the same time a root edge and a
distinguished vertex). \subsubsection*{\bf{The BDFG construction}:}
\begin{itemize}
  \item[\bf{(1)}]{\bf orientation and labelling.}
First, we orient every edge of $\mathfrak{m}$ such as it has a black face on its
right. Then, we label each vertex $v$ of $\mathfrak{m}$ by the minimum number of
oriented edges needed to reach it from the pointed vertex. Observe that along an
oriented edge, the label can either increase of $1$, either
decrease by any nonnegative number.
%\item[]The
%increment of the label along $e$ of $e$ (so the type belongs to
%$\llbracket-1,+\infty)$).
 \item[\bf{(2)}]{\bf  local construction.}
First, inside each face of $\mathfrak{m}$, we add a new vertex of the colour of
the face. Then, inside each white face $F$ of $\mathfrak{m}$, and
for all edge $e$ adjacent to $F$, we procede to the following construction (see
Figure \ref{fig:BDFGrules}):
\begin{itemize}
\item
 if the label increases by $1$ along $e$, we  add a new edge between the
 unlabelled white vertex at the center of $F$ and the extremity of $e$ of greatest label.
\item
 if the label decreases by $\tau\geq 0$ along $e$, we add an new edge between
 the two central vertices lying at the centers of the two faces separated by 
$e$. Moreover, we mark each side of this is edge with a \emph{flag}, which is
itself labelled by the label in $\mathfrak{m}$ of the corresponding extremity
of $e$, as in Figure~\ref{fig:BDFGrules}.
\end{itemize}
\item[\bf{(3)}]{\bf erase original edges.} We let $\bar{\mathfrak{m}}$
be the map obtained by erasing all the original edges of $\mathfrak{m}$ and the
 pointed vertex $v_0$ (i.e. the consisting of all the new  vertices and
 edges added in the construction, and all the original vertices of the map
 except the pointed vertex). 
 \item[\bf{(4)}]{\bf choose a root  and shift labels}. We define the root of
 $\bar\mathfrak{m}$ as the edge associated to the root edge of
 $\mathfrak{m}$ in the above construction ; we orient it such that it leaves 
 a white unlabelled vertex. The \emph{root label} is either the label of the
 only labelled vertex adjacent to the root edge (if it exists), either the label
 of the flag situated on the left of the root edge. We now translate all the
 labels in $\bar\mathfrak{m}$ by the opposite of the root label, so that the
 new root label is $0$: we let $Mob(\mathfrak{m})$ be the map obtained at this
 step. A planar example is shown on Figure~\ref{fig:planarbij}.
\end{itemize}
\begin{figure}[h]
\centerline{\includegraphics[scale=1.2]{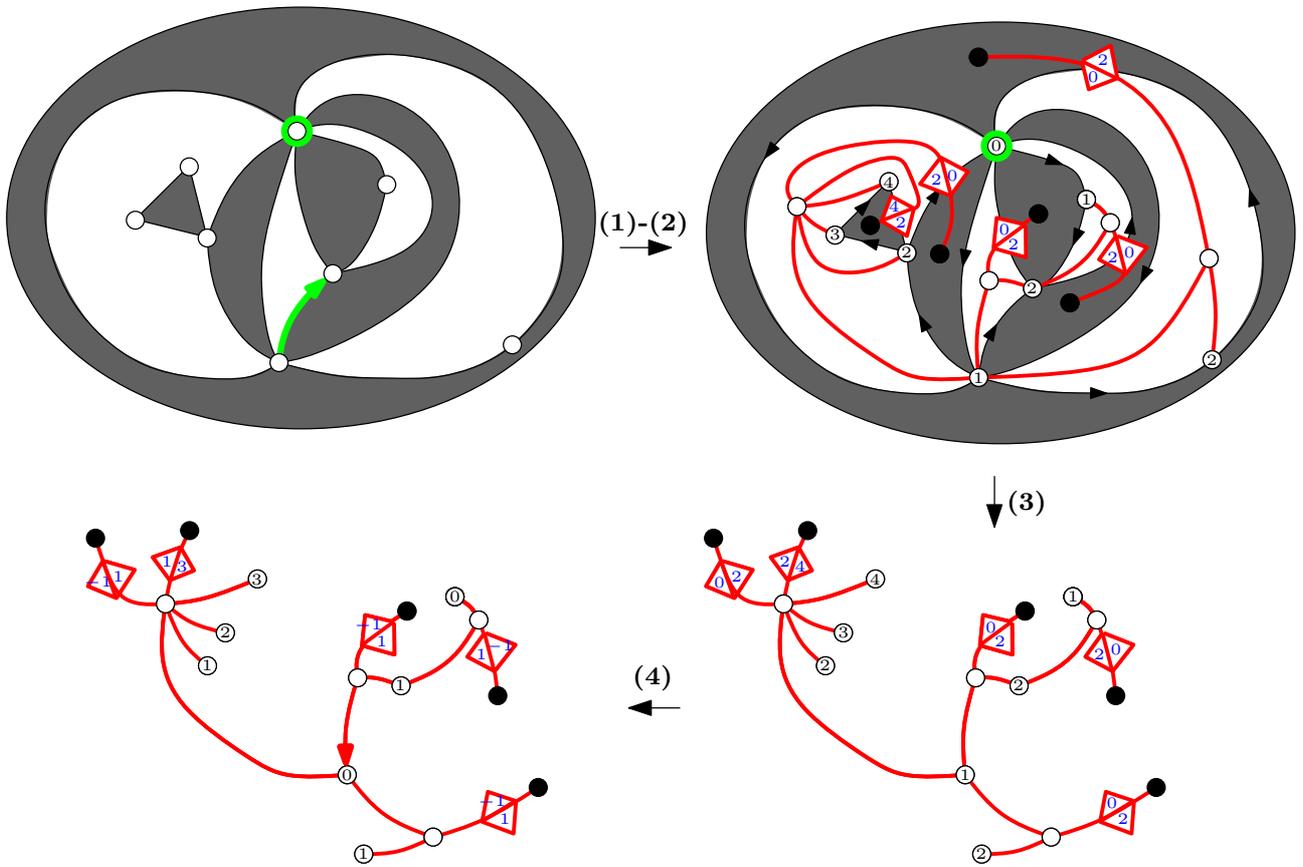}}
\caption{A rooted and pointed $3$-constellation on the sphere, and its
associated mobile.}
\label{fig:planarbij}
\end{figure}

Recall that a $g$-tree is a map on $\mathcal{S}_g$ which has only one face. In
the planar case, from Euler characteristic formula, this is equivalent to the
classical graph-theoretical definition of a tree. However, in positive genus, a $g$-tree
always has cycles, and therefore will never be a tree, in the graph sense.
We have:
\begin{lemma}
\label{lemma:welldef}
$\mathrm{Mob}(\mathfrak{m})$ is a well-defined map on $\mathcal{S}_g$, and is moreover
a $g$-tree.
\end{lemma}
\begin{proof}
Our proof follows the arguments of \cite{ChMaSc}.
We let $\mathfrak{m}'$ be the map consisting of 
the original map $\mathfrak{m}$ and all the new vertices and edges added in the
previous construction ; to avoid edge-crossings, each time a flagged edge of
$Mob(\mathfrak{m})$ crosses an edge of $\mathfrak{m}$, we split those two edges
in their middle, and we consider the pair of flags lying in the middle of the
flagged edge as a tetravalent vertex of $\mathfrak{m}'$, linked to the four
ends created by the edge-splitting.

It is clear from the construction rules that each black or white unlabelled
vertex is adjacent to at least one flagged edge, so that
$\mathfrak{m}'$ is a well defined connected map of genus $g$.
We now let $\widehat{\mathfrak{m}'}$ be the dual map of $\mathfrak{m}'$, and 
$\mathfrak{t}$ be the submap of $\widehat{\mathfrak{m}'}$ induced by the set of
edges of $\mathfrak{m}'$ which are dual edges of original edges of $\mathfrak{m}$. 
We now examine the cycles of $\mathfrak{t}$.

By convention, we orient each edge of $\mathfrak{t}$ as follows: if the edge lies between a vertex and 
a flag, then we orient it in such a way that it has the flag on its left. If it
lies between two vertices, then we orient it in such a way that is has the
vertex of greatest label on its left. 
Then by the construction rules (see Figure~\ref{fig:nocycle}) each
face of $\mathfrak{m}'$ carries an unique outgoing edge of $\mathfrak{t}$.
Hence, if $\mathfrak{t}$ contains a cycle of edges, it is in fact an oriented cycle. 
Moreover, when going along an oriented  cycle of edges of $\mathfrak{t}$, the
label present at the right of the edge cannot increase (as seen on checking the
different cases on Figure~\ref{fig:nocycle}). Hence this label is constant
along the cycle, and looking one more time at the different cases on
Figure~\ref{fig:nocycle}, this is possible only if the cycle encircles a single vertex. 
Such a vertex cannot be incident to any vertex with a smaller label (otherwise,
by the construction rules, an edge of $Mob(\mathfrak{m})$ would cut the cycle),
which implies by definition of the labelling by the distance that the encycled vertex 
is the pointed vertex $v_0$.

Hence $\mathfrak{t}$ has no other cycle than the cycle encycling $v_0$. This
means that, after removing $v_0$ and all the original edges of $\mathfrak{m}$, one does not create any non simply connected face, and that
$\mathrm{Mob}(\mathfrak{m})$ is a well defined map of genus $g$ (for a detailed
topological discussion of this implication, see the appendix in \cite{ChMaSc}).

Finally, let $b$ (resp. $w$) be the number of black (resp. white) faces of $\mathfrak{m}$,
and $v$ (resp. $e$) be its number of vertices (resp. edges). Then by Euler
characteristic formula, one has: $$
(b + w)  + v = e + 2 - 2g
$$ 
Now, by construction, $\mathrm{Mob}(\mathfrak{m})$ has $e$ edges and $b+w+v-1$ vertices. Hence
applying Euler characteristic formula to $\mathrm{Mob}(\mathfrak{m})$ shows that it has exactly one face, i.e. that it
is a $g$-tree.
\end{proof}
\begin{figure}[t]
\centerline{\includegraphics[scale=1.0]{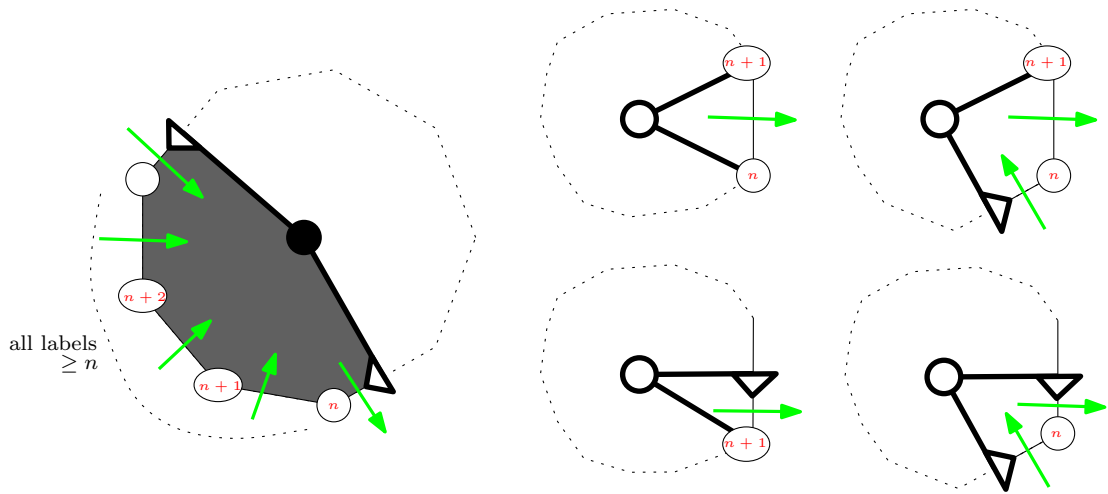}}
\caption{
A typical black face, and the four types of white faces of $\mathfrak{m}'$.
When a cycle of arrows crosses a face, the label at its right cannot increase.
Moreover, it remains constant if and only if it turns around a single vertex.}
\label{fig:nocycle}
\end{figure}

\subsection{From mobiles to maps}

Our definition of a mobile is taken from \cite{BDFG}:
\begin{definition}\label{def:mobile}
A $g$-mobile is a rooted $g$-tree $\mathfrak{t}$ such that:
\begin{itemize}
\item[\bf i.] $\mathfrak{t}$ has vertices of three types: unlabelled ones, which
can be black or white, and labelled ones carrying integer labels.
%\item the labels are $\geq 1$, and there is at least one vertex labelled $1$.
\item[\bf ii.] edges can either connect a labelled vertex to a white unlabelled
vertex, either connect two unlabelled vertices of different color. The edges of
the second type carry on each side a flag, which is itself labelled by an integer.
\item[\bf iii-w.] when going clockwise around a white unlabelled vertex:
 \begin{itemize}
  \item a vertex labelled $l$ is followed by a label $l-1$ (either vertex or
  flag).
  \item two successive flags of labels $l$ and $l'$ lying on the
  same edge satisfy $l' \geq l$ ;the second flag
  is followed by a label $l'$ (either vertex or flag).
 \end{itemize} 
\item[\bf iii-b.] when going clockwise around a black unlabelled vertex,
two flags of labels $l$ and $l'$ lying on the same side of an edge satisfy 
     $l' \leq l$ ; the second flag is followed by a flag labelled
     $\geq l'$.
     \item[\bf iv.] The root edge is oriented leaving a white unlabelled
     vertex. The root label (which is either the label of the labelled vertex
     adjacent to the root, if it exists, either the label of the flag present
     on its left side) is equal to $0$.
\end{itemize}
\end{definition}

One easily chack that the BDFG construction leads to a map that satisfies the
conditions above.
Hence, thanks to the previous lemma, for every eulerian map $\mathfrak{m}$, $\mathrm{Mob}(\mathfrak{m})$
is a $g$-mobile. We now describe the reverse construction, that associates a eulerian map to any $g$-mobile.
This construction takes place inside the unique face of $\mathrm{Mob}(\mathfrak{m})$. In particular, we want to insist
on the fact that all the work specific to the non planar case has been done when proving that
$\mathrm{Mob}(\mathfrak{m})$ is a $g$-tree. Until the rest of this section, everything is similar to the planar case. 
For this reason, we refer the reader to \cite{BDFG} for proofs.
Let $\mathfrak{t}$ be a $g$-mobile. The closure of $\mathfrak{t}$ is defined as follows:
\subsubsection*{\bf{Reverse construction:}}%
\begin{itemize}
\item[\bf{(0)}] Translate all the labels of $\mathfrak{t}$ by the same integer
in such a way that the minimum label is either a flag of label $0$, either a
labelled vertex of label $1$.
\item[\bf{(1)}]
 Add a vertex of label $0$ inside the unique face of $\mathfrak{t}$. Connect it by an edge
 to all the labelled corners of $\mathfrak{t}$ of label $1$, and to all the
 flags labelled $0$.
\item[\bf{(2)}]
 Draw an edge between each labelled corner of $\mathfrak{t}$ of label $n\geq 2$ and its
 succesor, which is the first labelled corner or flag with label $n-1$ encountered when going counterclockwise around $\mathfrak{t}$.
\item[\bf{(3)}]
 Draw an edge between each flag of label $n$ and its succesor, which is the first
 labelled corner or flag with label $n$ encountered when going counterclockwise around $\mathfrak{t}$.
\item[\bf{(4)}]
 Remove all the original edges and unlabelled vertices of $\mathfrak{t}$. 
\end{itemize}
We call $\mathrm{Map}(\mathfrak{t})$ the map obtained at the end of this
construction. The root of $Map(\mathfrak{t})$ is either the root joining the
endpoint of the root of $\mathfrak{t}$ to its succesor (if it is labelled),
either the edge corresponding to the flags lying on the root edge.
The fact that this construction is reciprocal to the previous one is
proved in the planar case in \cite{BDFG}, but, as we already said, every 
argument stay valid in higher genus. Hence we have: \begin{lemma}[\cite{BDFG}]
\label{lemma:BDFG}
For every eulerian map $\mathfrak{m}$, one has: 
$\mathrm{Map}(\mathrm{Mob}(\mathfrak{m}))=\mathfrak{m}$
For every $g$-mobile $\mathfrak{t}$, one has:
$\mathrm{Mob}(\mathrm{Map}(\mathfrak{t}))=\mathfrak{t}
$
\end{lemma}
This proves:
\begin{theorem}
\label{thm:bij}
The application $\mathrm{Mob}$ defines a bijection between the set of
rooted and pointed eulerian maps of genus $g$ with $n$ edges and the set of
$g$-mobiles with $n$ edges. This bijection sends a map which has $n_i$
white faces of degree $i$ for all $i$, $b$ black faces, and $v$
vertices to a mobile which has $n_i$ white unlabelled vertices of degree $i$ for
all $i$, $b$ black unlabelled vertices and $v-1$
labelled vertices.
\end{theorem}

\subsection{$m$-constellations and $m$-hypermaps.}

Mobiles obtained from a $m$-hypermaps form a subset of the set of all mobiles,
and satisfy additionnal property. To keep the terminology reasonable, we make
the following convention:  \\
{\bf Convention:}
In the rest of the paper, the word \emph{mobile} will refer only to mobiles
which are associated to $m$-hypermaps of genus $g$ by the Bouttier-Di
Francesco-Guitter bijection.

Let $\mathfrak{m}$ be a rooted and pointed $m$-hypermap, with vertices labelled
by the distance from the pointed vertex. We define the \emph{increment} of an
(oriented) edge as the label of its endpoint minus the label of its origin ;
since all black faces have degree $m$, by the triangle inequality, all increments are in
$\llbracket -1, m-1 \rrbracket$. More, if if a black face is adjacent to an
edge $e$ of increment $m-1$, and since the sum of the increments is null along a
face, then its $m-1$ other edges must have type $-1$. Hence, the black
unlabelled vertex of the corresponding mobile has degree $1$: it is connected
only to the flagged edge corresponding to $e$.

Now, let $\mathfrak{t}$ be a mobile. The \emph{increment} of a flagged edge is
the increment of the associated edge in the corresponding $m$-hypermap: it is
therefore the difference of the labels of the two flags, clockwise around the
white unlabelled vertex. All black unlabelled vertices of degree $1$ are
linked to a flagged edge of increment $m-1$.
% We now erase all those black
%vertices of degree $1$, and the corresponding flagged edges of increment $m-1$:
%we obtain a tree $\bar\mathfrak{t}$ that we call the \emph{simplified mobile} 
%of $\mathfrak{t}$. Observe that $\mathfrak{t}$ can easily be reconstructed from
%$\bar\mathfrak{t}$: around any white unlabelled vertex, the positions where
%the black vertices must be re-inserted are precisely the ones where the label
%increases between two  clockwise-consecutive labels.

Now, observe that a $m$-hypermap is a $m$-constellation if and only if the
labelling of its vertices by the distance from the pointed vertex, taken modulo $m$, realizes
the property {\bf iii} of the definition of a constellation. Indeed, in a
$m$-constellation, the difference modulo $m$ between the distance labelling and
any labelling realizing property {\bf iii} is constant on a geodesic path of
oriented edges from the pointed vertex to any vertex, since both increase by
$1$ modulo $m$ at each step.
% Indeed, a
%$m$-hypermap is a $m$-constellation if and only if the length on any path of
%oriented edges linking two vertices $v$ and $w$ is congruated modulo $m$ to the
%difference of their labels by the distance.
Hence, all the edges of a $m$-constellation have an increment which is either
$-1$, either $m-1$. 
This gives:
\begin{lemma}
Let $\mathfrak{m}$ be a rooted and pointed $m$-hypermap, with vertices labelled
by the distance from the pointed vertex.
Then $\mathfrak{m}$ is a $m$-constellation
if and only if one of the following two equivalent properties holds:
\begin{itemize}
  \item all its edges have increment $-1$ or $m-1$
  \item all the black unlabelled vertices of its mobile have degree $1$
%  \item its simplified mobile has no black unlabelled vertices.
\end{itemize}
\end{lemma}
In particular, $\mathfrak{m}$ is a $m$-constellation if and only if, clockwise
around any black face, the label increases by $1$ exactly $m-1$ times, and
decreases by $m-1$ exactly one time.

\section{The building blocks of mobiles: elementary stars and cells.}
\label{sec:stars}

\subsection{Elementary stars.}

\begin{figure}[h]
\centerline{\includegraphics[scale=1]{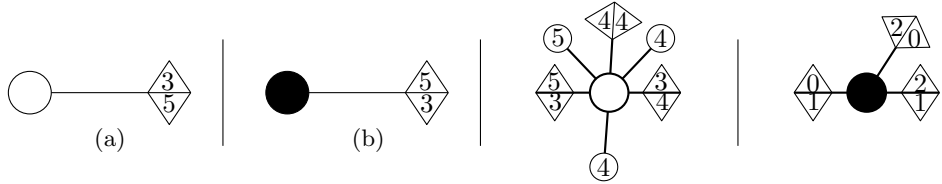}}
\caption{(a) a white split-edge of type $3$; (b) a black split edge of type
$3$ ; (c) a white elementary star; (d) a black elementary star.}
\label{fig:stars}
\end{figure}

We now define what the building blocks of mobiles are.

\begin{definition}(see Figure~\ref{fig:stars})
A \emph{white split-edge} is an edge that links a white
unlabelled vertex to a pair formed by two flags, each one lying on one side of the edge, as in
Figure~\ref{fig:stars}. Each flag is labelled by integer. If those labels are
$l_1$ and $l_2$, in clockwise order around the unlabelled vertex, the quantity
$l_2-l_1+1$ is called the \emph{type} of the split-edge. The same definition
holds for \emph{black split-edges}, but in this case the type is defined as
$l_1-l_2-1$.

A \emph{white elementary star} is a star formed by a central white unlabelled
vertex, which is connected to a certain number of labelled vertices, and to a
certain number of white split edges, and that satisfies the property {\bf
iii-w} of Definition~\ref{def:mobile}. Elementary stars are considered up to
translation of the labels. 

The same definition holds for black elementary stars, up to replacing ``white''
by ``black'' and property {\bf iii-w} by property {\bf iii-b}.
\end{definition}
The following lemma will be extremely usefull:
\begin{lemma}
\label{lemma:usefull}
Let $\mathfrak{s}$ be an elementary white star of degree $km$.
Assume that $\mathfrak{s}$ has $r$ split-edges, and let
$\tau_1,\tau_2,\ldots,\tau_r$ be their types. Then we have:
$$
\sum_{i=1}^r \tau_i = k m.
$$
\end{lemma}
\begin{proof}
We number the flags from $1$ to $r$, in clockwise order, starting anywhere. We
let $l_i$ and $l_i'$ be the labels carried by the $i$-th flag, in clockwise
order, so that the corresponding type is $\tau_i=l_i'-l_i+1$. By the
property {\bf iii-w}, the label decreases by one after each labelled vertex, so
that $l_{i}'-l_{i+1}$ is exactly the number of labelled vertices between the
$i$-th and $i+1$-th flags (with the convention that the $r+1$-th flag is the
first one). Hence the total degree of $\mathfrak{s}$ is: 
$$r+\sum_{i=1}^r (l_{i}'-l_{i+1}) = r+\sum_{i=1}^r (l_{i}'-l_{i}) = 
\sum_{i=1}^r \tau_i$$
which yields the result.
\end{proof}
\begin{remark}
If $\mathfrak{s}$ is a black elementary star which is present a mobile
$\mathfrak{t}$ such that $Map(\mathfrak{t})$ is a $m$-hypermap,
then the conclusion of the lemma also holds, with $k=1$. 
Indeed, if $l_1, .. , l_m$ is the clockwise sequence of the distance
labels around the corresponding black face of $Map(\mathfrak{t})$, then 
$\sum_{i=1}^m\tau_i = \sum_{i=1}^m (l_{i+1}-l_i+1)=m$.
\end{remark}

\begin{definition}
A $m$-walk of length $l$ is a $l$-uple of integers 
$(n_1,\ldots,n_l)\in\llbracket-1,m-1\rrbracket^l$ such that 
$\sum n_i=0$. A circular $m$-walk of length $l$ is a $m$-walk of length $l$
considered up to circular permutation of the labels (i.e. an orbit under the
action of the cyclic group $\mathbb{Z}_l$ on the indices).
\end{definition}

We now explain how to associate a $m$-walk to an elementary star.
Let $\mathfrak{s}$ be a white elementary star with degree multiple of $m$. We
read clockwise the sequence of labels vertices and split-edges around the central
vertex. We interpret labelled vertices as a number $-1$, and split-edges of type
$\tau$ as a number $\tau-1$. We obtain a sequence of integers $(n_1,\ldots,n_l)$
defined up to circular permutations.
\begin{lemma}\label{lemma:mwalk}
For each $l$ multiple of $m$, the construction above defines a bijection
between white elementary stars of degree $l$ and circular $m$-walks of length
$l$.
\end{lemma}
\begin{proof}
It follows from the property {\bf iii-w} that the walk associated to a white
elementary star is indeed a $m$-walk. Conversely, given a $m$-walk of length
$l$, and interpreting steps $-1$ as labelled vertices, and steps $\tau-1$ as
split-edges of type $\tau$, one reconstructs an elementary white star, which is
clearly the only one from which the construction above recovers the original
walk.
\end{proof}

\begin{definition}
We say that a split-edge is \emph{special} if its type is not equal to $m$. A
star is \emph{special} if it contains at least one split-edge, and
\emph{standard} otherwise.
\end{definition}

\subsection{Cells and chains of type $0$.}

\begin{figure}[h]
\centerline{\includegraphics[scale=0.8]{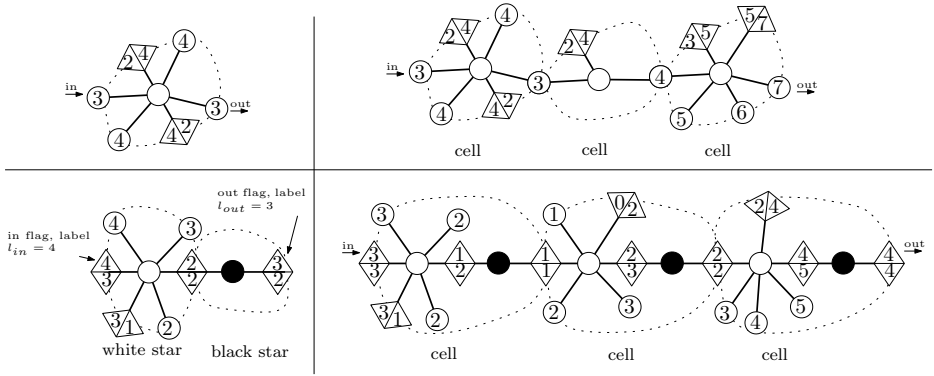}}
\caption{Four examples in the case $m=3$. Up, a cell of type $0$, and a chain of
type $0$ ; down, a cell of type $1$, and a chain of type $2$.}
\label{fig:chains}
\end{figure}

\begin{definition}(see Figure~\ref{fig:chains})
A \emph{cell of type $0$} is a standard elementary white star of degree multiple
of $m$, which carries two distinguished labelled vertices: 
the \emph{in} one and the \emph{out} one.
\\ The \emph{increment} of a cell of
type $0$ is the difference $l_\outv-l_\inv$ of the labels of its out and in vertices. Its \emph{size} is its number of 
split-edges, and its \emph{total degree} is its degree as an elementary star
(i.e. the degree of the central vertex).

A \emph{chain of type $0$} is a finite sequence of cells type $0$. Its
size and increment are defined additively from the size and increment of the
cells it contains. Its \emph{in} vertex (resp \emph{out} vertex) is the in
vertex of its first cell (resp. out vertex of its last cell).
\end{definition}
On pictures, to draw a chain of type $0$, we identify the out vertex of each
cell with the in vertex of the following one, as in Figure~\ref{fig:chains}.
Observe that, from the previous lemma, the total degree of a cell of type $0$ equals $m$ times its size.
Consequently, the total number of corners of the chain adjacent to a
labelled vertex equals $(m-1)$ times its size.

\subsection{Cells and chains of type $\tau\in\llbracket 1,m-1 \rrbracket$.}

\begin{definition}(see Figure~\ref{fig:chains})
Let $\tau\in\llbracket 1,m-1 \rrbracket$.
A \emph{cell of type $\tau$} is a pair
$(\mathfrak{s}_1,\mathfrak{s}_2)$ where:
\\{\bf -} $\mathfrak{s}_1$ is an elementary white star, with exactly two
special split-edges: the \emph{in} one, of type $\tau$, and the \emph{out}
one, of type $m-\tau$.
\\{\bf -} $\mathfrak{s}_2$ is an elementary black star, with exactly two
special split-edges: the \emph{in} one, of type $m-\tau$, and the \emph{out}
one, of type $\tau$.

 On pictures, we identify the two split-edges of type $m-\tau$, as in
Figure~\ref{fig:chains}. 
The \emph{in} split-edge of the cell is the in
split-edge of $\mathfrak{s}_1$, and its \emph{out} split-edge is the out
split-edge of $\mathfrak{s}_2$ ; the corresponding labels $l_\inv$ and
$l_\outv$ are defined with the convention of Figure~\ref{fig:chains}. The
\emph{increment} of the cell is the difference $l_\outv$ - $l_\inv$.

A chain of type $\tau$ is a finite sequence $\mathfrak{c}$ of cells of type
$\tau$. On pictures, we glue the flags of the out split-edge of a cell with the
flags of the in split-edge of the following cell, as in Figure~\ref{fig:chains}. The
\emph{increment} of the chain is the sum of the increment of the cells it
contains. We let $|\mathfrak{c}|$ denote the total number of labelled vertices
appearing in $\mathfrak{c}$. We also let $\langle\mathfrak{c}\rangle$ be the
total number of black vertices appearing in $\mathfrak{c}$ plus its total
number of split-edges of type $m$ (equivalently, $\langle
\mathfrak{c}\rangle$ is the total number of black vertices of $\mathfrak{c}$ if
one links each split-edge of type $m$ to a new univalent black vertex).
\end{definition}

\section{The full scheme of a mobile.}

In this section, we explain how to reduce mobiles of genus $g$ to a finite
number of cases, indexed by minimal objects called their \emph{full schemes}.
This is a generalisation of \cite{ChMaSc}.

\subsection{Schemes.}

\begin{definition}
A \emph{scheme of genus $g$} is a rooted map $\mathfrak{s}$ of genus $g$, which
has only one face, and whose all vertices have degree $\geq 3$. The set of schemes of genus
$g$ is denoted $\mathcal{S}_g$.
\end{definition}
Let $\mathfrak{s}$ be a scheme of genus $g$, and, for all $i\geq 3$, let $n_i$
be the number of vertices of $\mathfrak{s}$ of degree $i$. Then, by the
hand-shaking lemma, its number of edges is $\frac{1}{2}\sum in_i$, and
Euler Characteristic formula gives:
\begin{eqnarray}\label{eq:Eulerscheme}
\sum_{i\geq 3}\frac{i-2}{2}n_i = 2g-1.
\end{eqnarray}
Hence the sequence $(n_i)_{i\geq 3}$ can only take a finite number of values.
Since the number of maps with a given degree sequence is finite, this proves:
\begin{lemma}\cite{ChMaSc}
The set $\mathcal{S}_g$ of all schemes of genus $g$ is finite.
\end{lemma}
We now need a technical discussion that will be of importance later. We assume
that each scheme of genus $g$ carries an arbitrary orientation and labelling of
its edges, chosen arbitrarily but fixed once and for all. This will allow us to
talk about ``the $i$-th edge'' of a scheme, or ``the canonical orientation'' of
an edge, without more precision.
Our first construction is not specific to mobiles, and applies to all maps of
genus $g$ with one face (see Figure~\ref{fig:scheme}:
\begin{algorithm}[The scheme of $g$-tree
$\mathfrak{t}$.] Let $\mathfrak{t}$ be a $g$-tree. First, if $\mathfrak{t}$
contains a vertex of degree $1$, we erase it, together with the edge it is connected to. We then
repeat this step recursively until there are no vertices of degree $1$ left. We
are left with a map $\mathfrak{c}$, which we call \emph{the core of
$\mathfrak{t}$}. If the original root of $\mathfrak{t}$ is still present in the
core, we keep it as the root of $\mathfrak{c}$. Otherwise, the root is present
in some subtree of $\mathfrak{t}$ which is attached to $\mathfrak{c}$ at some
vertex $v$: we let the root of $\mathfrak{c}$ be the first edge 
of $\mathfrak{c}$ encountered after that subtree when
turning clockwise around $v$ (and we orient it
leaving $v$).

Now, in the core, vertices of degree $2$ are organised into maximal
paths connected together at vertices of degree at least $3$. We now replace
each of these paths by an edge: we obtain a map $\mathfrak{s}$, which has only
vertices of degree $\geq 3$. The root of $\mathfrak{s}$ is the
edge  corresponding to the path that was carrying the root of $\mathfrak{c}$
(with the same orientation). We say that $\mathfrak{s}$ is \emph{the scheme of
$\mathfrak{t}$}. The vertices of $\mathfrak{t}$ that remain vertices of
$\mathfrak{s}$ are called \emph{the nodes of $\mathfrak{t}$}.
\end{algorithm}

\begin{figure}[h]
\centerline{\includegraphics[scale=0.8]{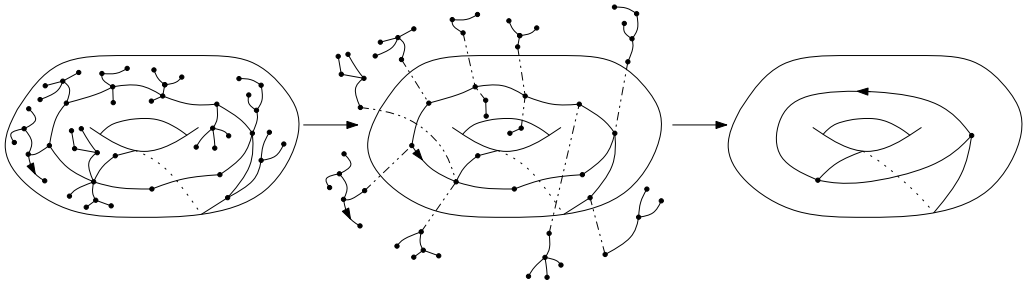}}
\caption{From a $1$-tree to its scheme.}
\label{fig:scheme}
\end{figure}

\subsection{The superchains of a mobile.}

Let $\mathfrak{t}$ be a mobile whose scheme $\mathfrak{s}$ has $k$ edges. Each
edge of $\mathfrak{s}$ corresponds to a path of vertices of degree $2$ of the core. For $i=1..k$, we let $\mathfrak{p}_i$ be
the path corresponding to the $i$-th edge of
$\mathfrak{s}$, oriented by the canonical orientation of this edge (observe
that each node is the extremity of several paths). A priori, $\mathfrak{p}_i$
can contain labelled vertices, black or white unlabelled vertices, and flagged or
unflagged edges. We have the following important lemma:
\begin{lemma}
\label{lemma:nospecial}
All the special flagged edges of $\mathfrak{t}$ lie on the paths
$\mathfrak{p}_i$, $i=1..k$.
\end{lemma}
\begin{proof}
Assume that there is a special flagged edge $e_0$ in
$\mathfrak{t}\setminus\mathfrak{c}$: $e_0$ belongs to a subtree $\mathfrak{t}'$ 
that has been detached from $\mathfrak{t}$ during the construction of its core.
$e_0$ is connected to two unlabelled vertices, one of them, say $v$, being the
farthest from $\mathfrak{c}$. Now, by Lemma~\ref{lemma:usefull}, an unlabelled vertex (black or white) of
$\mathfrak{t}$ cannot be connected to exactly one special flagged edge. Hence
$v$ is connected to another special edge $e_1$. Repeating recursively this
argument, one constructs an infinite sequence of special edges
$e_0,e_1,\ldots$. All these special edges belong to the subtree
$\mathfrak{t'}$, so that the sequence cannot form a cycle: this implies that
these edges are all distinct, which is impossible since a mobile has a finite
number of edges.
\end{proof}

\begin{figure}[h]
\centerline{\includegraphics[scale=1]{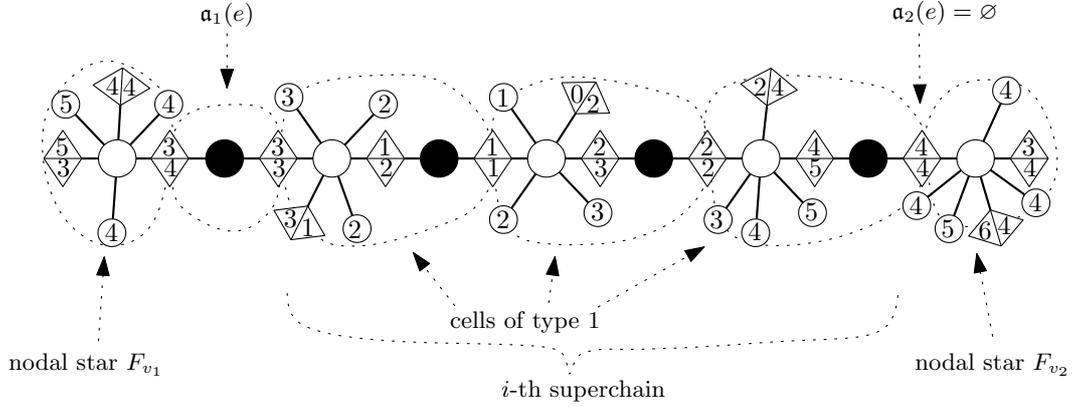}}
\caption{A typical superchain of type $2$, in the case $m=3$. It has two nodal
star, only one correcting term $\mathfrak{a}_1(e)$ ($\mathfrak{a}_2(e)$ is empty), and the
superchain itself is formed of three consecutive cells of type $1$.}
\label{fig:superchain}
\end{figure}

Each unlabelled vertex of $\mathfrak{p}_i$ was, in the
original mobile  $\mathfrak{t}$, at the center of an elementary star. We now
re-draw all these elementary stars around each unlabelled vertex of $\mathfrak{p}_i$, 
as on Figure~\ref{fig:superchain}. If the extremities of $\mathfrak{p}_i$ are
unlabelled vertices, we say that the corresponding stars are \emph{nodal stars}
of $\mathfrak{t}$.  For the moment, we remove the nodal stars, if they exist: we
obtain a (eventually empty) sequence $\mathfrak{s}_1,\ldots,\mathfrak{s}_l$ of
successive stars. We now have to distinguish two cases.

\subsubsection*{case 1: $\mathfrak{p}_i$ contains no special flagged edge.}
In this case, $\mathfrak{p}_i$ is made of succession of edges linking white
unlabelled vertices to labelled vertices (since the only remaining case,
flagged edges of type $m$, are only linked to univalent black vertices and then
cannot be part of the core). Consequently, the sequence
$\mathfrak{s}_1,\ldots,\mathfrak{s}_l$ is a sequence of white
elementary stars, with no special flagged edges, glued together at labelled
vertices, i.e. a chain of type $0$ in the terminology of the
preceding section. 
We say that $(\mathfrak{s}_1,\ldots,\mathfrak{s}_l)$ is the $i$-th superchain of $\mathfrak{t}$.

\subsubsection*{case 2: $\mathfrak{p}_i$ contains at least one special flagged
edge} 
In this case, we will also show that our path reduces to a sequence
of cells. 
First, from Lemma~\ref{lemma:usefull}, an unlabelled vertex cannot be adjacent
to exactly one special edge. Now, from Lemma~\ref{lemma:nospecial}, an unlabelled
vertex of $\mathfrak{p}_i$ which is not one of its extremities cannot be
adjacent to more than $2$ special edges in $\mathfrak{t}$. Hence such a vertex
is adjacent either to $0$ or $2$ special edges. Hence the set of special
flagged edges of $\mathfrak{p}_i$ forms itself a path with the same extremities
as $\mathfrak{p}_i$, i.e. is equal to $\mathfrak{p}_i$. In other terms: all the
edges of $\mathfrak{p}_i$ are special flagged edges.\\
We now consider the sequence of stars $\mathfrak{s}_1,\ldots,\mathfrak{s}_l$.
If the first star of the sequence is black, we call it
$\mathfrak{a}_1(i)$ and we remove it (otherwise we put formally
$\mathfrak{a}_1(i)=\varnothing$). Similarly, if the last star is white, we call
it $\mathfrak{a}_2(i)$ and we remove it. We now have a sequence of alternating color stars $(\mathfrak{s}'_1,\ldots,\mathfrak{s}'_{l'})$ that
begins with a white star and ends with a black one. From what we just said, all these stars are elementary
stars with exaclty two special flagged edges, glued together at these flagged
edges. Since the sequence is ordered, we can talk of the ingoing and outgoing
special edge of each of these stars. Now, let $\tau$ be the type of the ingoing
special edge of $\mathfrak{s}'_1$. By Lemma~\ref{lemma:usefull}, the type of
its outgoing special edge is $m-\tau$. Now, this flagged edge is also the
ingoing edge of the black star $\mathfrak{s}'_2$, and applying
Lemma~\ref{lemma:usefull} again, the type of the outgoing edge of
$\mathfrak{s}'_2$ is $m-(m-\tau)=\tau$. Consequently,
$(\mathfrak{s}'_1,\mathfrak{s}'_2)$ is a cell of type $\tau$, in the sense of
the previous section. Applying recursively the argument, each pair
$(\mathfrak{s}'_{2q-1},\mathfrak{s}'_{2q})$ is a cell of type $\tau$. The
sequence $(\mathfrak{s}'_1,\ldots,\mathfrak{s}'_l)$ is therefore a chain of type
$\tau$, which we call the $i$-th superchain of $\mathfrak{t}$.

In the two cases above, we have associated to the $i$-th edge $e$ of
$\mathfrak{s}$ a chain, which we called the $i$-th superchain of $\mathfrak{t}$.
We now define the \emph{type} of $e$ as the type of this chain, and we note it
$\tau(e)$. 

By convention, if the $i$-th edge has type $0$, we put
$\mathfrak{a}_1(i)=\mathfrak{a}_2(i)=\varnothing$.

\subsection{Typed schemes and the Kirchoff law.}
\label{subsec:Kirchoff}
Let $v$ be a node of $\mathfrak{t}$. If $v$ is labelled, then it is
connected to no flagged edge (since flagged edges only connect
unlabelled vertices). Hence all the paths $\mathfrak{p}_i$'s that are meeting at $v$
correspond to case 1 above, or equivalently, all the edges of $\mathfrak{s}$
meeting at $v$ are edges of type $0$.

On the contrary, assume that $v$ is unlabelled. Let $e$ be an edge of
$\mathfrak{s}$ adjacent to $v$, of type $\tau(e)\neq0$.
and let $\mathfrak{p}_i$ be the corresponding path of the core. We let $\tilde\tau(e)$
be the type of the split-edge of $\mathfrak{p}_i$ which is adjacent to $v$.
It follows from the construction rules of the scheme that if $v$ is
\emph{white}, then one has $\tilde\tau(e)=\tau(e)$ if $e$ is incoming at $v$
and $\tilde\tau(e)=m-\tau(e)$ if it is outgoing. On the contrary, if $v$ is
\emph{black} then $\tilde\tau(e)=m-\tau(e)$ if $e$ is incoming and
$\tilde\tau(e)=\tau(e)$ if it is outgoing. Now, in both cases, 
Lemma~\ref{lemma:usefull} or the remark following it give: $\sum_{e\sim v}
\tilde\tau(e)=0\mbox{ mod }m$. 

In all cases, we have therefore:
\begin{proposition}[Kirchoff law]\label{prop:Kirchoff}
Let $v$ be a vertex of $\mathfrak{s}$. We have:
\begin{eqnarray}\label{eq:kirchoff}
\sum_{e \scriptstyle \ outgoing}\tau(e) - \sum_{e
 \ \scriptstyle ingoing}\tau(e) =0 \ \mathrm{mod} \ m
\end{eqnarray}
\end{proposition}
This leads to the following definition:
\begin{definition}
Let $\mathfrak{s}$ be a scheme of genus $g$. A \emph{typing of $\mathfrak{s}$} 
is an
application $\tau : \{\mbox{edges of }\mathfrak{s}\}\rightarrow \llbracket 0 ,
m-1 \rrbracket$ that satisfies Equation~\ref{eq:kirchoff} around each vertex.\\
A \emph{typed scheme} is a pair $(\mathfrak{s},\tau)$ formed by a scheme and one
of its typings.\\
If $\mathfrak{s}$ is the scheme of a mobile $\mathfrak{t}$, and $\tau$ is
the application that associates to each edge of $\mathfrak{s}$ the type of its
corresponding superchain, we say that $(\mathfrak{s},\tau)$ is \emph{the typed
scheme of $\mathfrak{t}$}.
\end{definition}
For future reference, we now state the following lemma, which is a key fact in
the proof of Theorem~\ref{thm:hypermaps}:
\begin{lemma}
\label{lemma:dimension}
Let $\mathfrak{s}$ be a scheme of genus $g$. Then $\mathfrak{s}$ has exactly
$m^{2g}$ different typings.
\end{lemma}
\begin{proof}
Observe that, if we identify $\llbracket 0, m-1 \rrbracket$  with
$\mathbb{Z}/m\mathbb{Z}$, the set of all valid typings of $\mathfrak{s}$ is a
$\mathbb{Z}/m\mathbb{Z}$ vector space. Actually, it coincides with the cycle space of $\mathfrak{s}$ in the sense of algebraic graph theory (see
\cite{Tutte:book} for an introduction to this notion). Now, it is
classical that the dimension of the cycle space of a connected graph equals its
number of edges minus its number of vertices plus $1$ (to see that, observe
that the complementary edges of any spanning tree form a basis of this space).
Now, since $\mathfrak{s}$ has one face, Euler characteristic formula gives: $$
\#\mbox{edges of }\mathfrak{s}-\#\mbox{vertices of }\mathfrak{s} = 2g-1.
$$
Hence the cycle space has dimension $2g$, and its cardinality is $m^{2g}$.
\end{proof}

\subsection{Nodal stars and decorated schemes.}

Let once again $v$ be a node of $\mathfrak{t}$. If $v$ is unlabelled, it is
located at the center of an elementary star $F_v$ (which, as we already said,
we call a nodal star). $F_v$ has a certain number of special split-edges, and a
certain number of distinguished labelled vertices, which are connected to the
paths $\mathfrak{p}_i$'s of $\mathfrak{t}$. We slightly abuse notations here,
and assume the notation $F_v$ denotes not only the elementary star itself, but
the elementary star together with those distinguished vertices and the
application that maps each distinguished vertex and split-edge of $F_v$ to the
corresponding half-edge of $\mathfrak{s}$.

In the case where $v$ is labelled, we put formally $F_v=\circ$, where $\circ$
may be understood as a single labelled vertex considered up to translation (so
that its label does not import).

Until the rest of the paper, if $\mathfrak{s}$ is a scheme of genus $g$, we
note $E(\mathfrak{s})$ and $V(\mathfrak{s})$ for the sets of edges and
vertices of $\mathfrak{s}$, respectively. If $|E(\mathfrak{s})|=k$, we will
sometimes identify $E(\mathfrak{s})$ with $\llbracket 1, k\rrbracket$.

\begin{definition}
We say that the quadruple 
$$
(\mathfrak{s},\tau,F,\mathfrak{a})=
\left(
\mathfrak{s},
(\tau(e))_{e\in E(\mathfrak{s})},
(F(v))_{v\in V(\mathfrak{s})},
(\mathfrak{a}_1(e),\mathfrak{a}_2(e))_{e\in E(\mathfrak{s})}\right)$$ is the
\emph{decorated scheme} of $\mathfrak{t}$.
\end{definition}

\subsection{The full scheme of a mobile.}
We now present the last step of the reduction of mobiles to elementary objects.
We assume that for each decorated scheme $(\mathfrak{s},\tau,F,\mathfrak{a})$,
and for each vertex $v$ of $\mathfrak{s}$, the star $F_v$ carries an arbitrary
but fixed labelled vertex or flag, chosen once and for all, that we call the
\emph{canonical element} of $v$. 

Now, let $\mathfrak{t}$ be a mobile, of decorated scheme
$(\mathfrak{s},\tau,F,\mathfrak{a})$. For each vertex $v$ of $\mathfrak{s}$, 
we let $l_v$ be the label in $\mathfrak{t}$ of the canonical element of $v$. We
now normalize these labels, so that they form an integer interval of minimum
$0$. Precisely, we let $M=\mathrm{card}\ \{l_v, \ v\in V(\mathfrak{s})\}-1$ and
$\lambda$ be the unique surjective increasing application $\{l_v, \ v\in
V(\mathfrak{s})\} \rightarrow \llbracket 0,M \rrbracket$.
\begin{definition}
We say that the quintuple $(\mathfrak{s},\tau,F,\mathfrak{a},\lambda)$ is the
\emph{full scheme} of $\mathfrak{t}$.
\end{definition}
In few words, the full scheme of $\mathfrak{t}$ contains five informations: the
combinatorial arrangement of the superchains, given by $\mathfrak{s}$ ; the
types of the superchains, given by $\tau$ ; the stars $F_v$ that lie on the
nodes of $\mathfrak{t}$; the (eventually trivial) stars $\mathfrak{a}_1(e)$
and $\mathfrak{a}_2(e)$ that ensure that each superchain of type $\neq 0$ begins
with a white star, and ends with a black one ; the relative order of the labels
of the canonical elements, given by $\lambda$.

Recall that the number of schemes of genus $g$, and the number of typings of a
given scheme, are finite. Moreover, since the set $D$ of allowed face degrees
is finite, there are only a finite number of elementary stars with total degree
in $mD$. Hence  $(F(v))_{v\in V(\mathfrak{s})}$ and 
$(\mathfrak{a}_1(e),\mathfrak{a}_2(e))_{e\in E(\mathfrak{s})})$ can only
take a finite number of values, and:
\begin{lemma}
The set $\mathcal{F}_{g}$ of all full schemes of genus
$g$ is finite.
\end{lemma}

Let $\mathfrak{f}=(\mathfrak{s},\tau,F,\mathfrak{a},\lambda)$ be a full scheme
of genus $g$. We say that a labelling $(l_v)_{v\in V(\mathfrak{s})}$ of its
canonical elements is compatible with $\mathfrak{f}$ if normalizing it to an
integer interval as we did above yields the application $\lambda$. We consider
compatible labellings up to translation, or equivalently, we assume that the
minimum $l_v$ is equal to $0$, so that all the compatible labellings are of the
form:
$$
l_v = \sum_{i=1}^{\lambda(v)} \delta_i \ \mbox{ for some
}\delta\in(\mathbb{N}_{>0})^{M}.$$
Assume that such a labelling has been fixed. To reconstruct a mobile, we have
to do the inverse of what precedes, and substitute a sequence of cells of the
good type along each edge of $\mathfrak{s}$. Observe that, for each edge $e$,
the increment $\Delta(e)$ of the superchain to be substituted to $e$ is
fixed by the choice of $(l_v)$. Precisely, let $e_+$ and $e_-$ be the
extremities of $e$, with the convention
$\lambda(e_+)\geq\lambda(e_-)$ (if $\lambda(e_+)=\lambda(e_-)$, any fixed
choice will be convenient). Then, up to the sign, we have
$\Delta(e)=l_{e_+}-l_{e_-}+a_{F,\mathfrak{a}}(e)$, where $a_{F,\mathfrak{a}}(e)$
is a correction term that does not depends on the $l_v$'s, and that accounts for 
the fact that superchains do not necessarily begin and end at
the canonical vertices. Precisely, $a_{F,\mathfrak{a}}(e)$ equals the
difference of the label of the canonical element of $e_+$ and the label of the
out vertex or flag of $\mathfrak{a}_2(e)$, from which one must substract the
corresponding quantity for $e_-$ (and it is important that these differences depend only on $F$ and
$\mathfrak{a}$). Putting
things in terms of the $\delta_i$'s, we can write:
$$\Delta(e)=a_{F,\mathfrak{a}}(e)+\delta_{e_-+1}+\ldots+\delta_{e_+}=a_{F,\mathfrak{a}}(e)+\sum_{j}A_{e,j}\delta_j
$$ where  for each edge $e$ and $j\in\llbracket 1,M\rrbracket$ we put
$A_{e,j}=\mathbbm{1}_{e_-<j\leq e_+}$.

%\section{Reconstructing all mobiles from their full schemes.}

\subsection{A non-deterministic algorithm.}
We consider the following non-deterministic algorithm:
\begin{algorithm}
\label{algo:reconstruct}
We reconstruct a mobile by the following steps:
\begin{itemize}
  \item[\bf 1.] we choose a full scheme
  $\mathfrak{f}=(\mathfrak{s},\tau,F,\mathfrak{a},\lambda)\in\mathcal{F}_g$.
  \item[\bf 2.] we choose a compatible labelling $(l_v)_{v\in V(\mathfrak{s})}$
  of $\mathfrak{f}$, or equivalently, a vector $\delta\in(\mathbb{N}_{>0})^{M}$.
  \item[\bf 3.] for each edge $e$ of $\mathfrak{s}$, we choose a chain of type
  $\tau(e)$. We then replace the edge $e$ by this chain, eventually preceded by
  the star $\mathfrak{a}_1(e)$ and followed by the star $\mathfrak{a}_2(e)$ if
  they are not empty.
  \item[\bf 4.] on each corner adjacent to a labelled vertex, we attach a
  planar mobile (which can eventually be trivial).
  \item[\bf 5.] we distinguish an edge as the root, and we orient it leaving a
  white unlabelled vertex.
  \item[\bf 6.] we shift all the labels in order that the root label is $0$.
\end{itemize}
\end{algorithm}
We have:
\begin{proposition}\label{prop:algo}
All mobiles of genus $g$ can be obtained by Algorithm~\ref{algo:reconstruct}.
More precisely, each mobile whose scheme has $k$ edges can be obtained in
exactly $2k$ ways by that algorithm.
\end{proposition}
\begin{proof}
The first statement follows by the decomposition we have explained until now:
we just have to re-add what we have deleted. Precisely, to reconstruct the
mobile $\mathfrak{t}$  from its full scheme, one can first recover the
labelling, then replace each edge by the corresponding superchain. Then, one
has to re-attach the planar trees that have been detached from $\mathfrak{t}$
during the construction of its core: this can be done at step {\bf 4}. Finally,
one obtains $\mathfrak{t}$ by choosing the right edge for its root, and shifting
the labels to fit the convention of the definition of a mobile.

Now, let us prove the second statement. It is clear that the only way to obtain
the mobile $\mathfrak{t}$ by different choices in the algorithm above is to
start at the beginning with a scheme which coincides with the scheme of
$\mathfrak{t}$ as un unrooted map, but may differ by the rooting.
Precisely, let us call a \emph{doubly-rooted mobile} a mobile whose scheme
carries a secondary oriented root edge. Clearly, a mobile whose scheme has $k$
edges corresponds to $2k$ doubly-rooted mobiles (since its scheme is already
rooted once, it has no symmetry). Now, Algorithm~\ref{algo:reconstruct} can be
viewed as an algorithm that produces a doubly-rooted mobile: the secondary root
of the scheme of the obtained mobile is given by the root of the scheme
$\mathfrak{s}$ chosen at step~{\bf 1} (we insist on the fact that the root of the starting scheme
$\mathfrak{s}$ has no reason to be the root of the scheme of the mobile obtained at the end).
Moreover, it is clear that each doubly-rooted mobile can be obtained in exaclty one way by the algorithm: the secondary root imposes the choice of the starting scheme $\mathfrak{s}$,
and after that all the choices are imposed by the strucure of $\mathfrak{t}$.
This concludes the proof of the proposition.
\end{proof}
\begin{remark}\label{rem:algobipartite}
Let us consider a variant of the algorithm, where at step {\bf 1}, we choose
only full schemes whose typing is identically $0$. Then
Proposition~\ref{prop:algo} is still true, up to replacing the word ''mobile''
by ``mobile associated with a $m$-constellation''. Indeed, a mobile is
associated to a $m$ constellation if and only if it has no special edge, and
the double-rooting argument in the proof of the proposition clearly works if we
restrict ourselves to this kind of mobiles.
\end{remark}

\section{Generating series of cells and chains}
Algorithm~\ref{algo:reconstruct} and Proposition~\ref{prop:algo} reduce the
enumeration of mobiles to the one of a few building blocks: schemes, planar
mobiles, cells and chains of given type. We now compute the corresponding
generating series.\\
{\bf Note:} In what follows, $m$ and $D$ are fixed. To keep things lighter,
the dependancy in $m$ and $D$ will many times be omitted in the notations.

\subsection{Planar mobiles.}
%We let $T(z):=\displaystyle\sum_{\mathfrak{t}\scriptsize\mbox{
%planar}}z^{|\mathfrak{t}|}$ be the generating series of planar mobiles by the
%size. For technical reasons, 
We let $\Tplanar(z)$ be the generating series, by the number of black
vertices, of planar mobiles whose root edge connects a white unlabelled vertex
to a labelled vertex. 
%Now, recall that all planar
%$m$-hypermaps are $m$-constellations, so that in a planar mobile all black
%vertices are univalent. Hence each white unlabelled vertex of total degree $mk$
%is connnected to exactly $(m-1)k$ labelled vertices and $k$ black unlabelled
%vertices. Hence we have:
%$$
%\Tplanar(z) = \frac{m-1}{m} T(z).
%$$
Observe that $\Tplanar$ is also the generating series of planar mobiles
which are rooted at a corner adjacent to a labelled vertex (for example, choose the
root corner as the first corner encountered clockwise after the root edge,
clockwise around the labelled vertex it is connected to).
Now, let $\mathfrak{t}$ be a planar mobile, whose root edge is connected to a
labelled vertex, and say
that the white elementary star containing the root-edge has total degree
$mk$. This star is attached to one planar mobile on each of its $(m-1)k$
labelled vertices ; each of these mobiles is naturally rooted at a labelled
corner. Moreover, given this star and the sequence of those $(m-1)k$ planar
mobiles, one can clearly reconstruct the mobile $\mathfrak{t}$. Finally, by
Lemma~\ref{lemma:mwalk}, the number of elementary white stars with 
total degree $mk$ and a distinguished edge connected to a labelled vertex
is equal to the number of $m$-walks of length $mk$ that begin with a
step $-1$, and have $(m-1)k$ steps $-1$ and $k$ steps $m-1$ in total,
which is ${mk -1 \choose k}$. This gives the equation:
\begin{eqnarray}
\label{eq:planar}
\Tplanar(z)= 1 + \sum_{k\in D} {mk -1  \choose k} z^k \Tplanar(z)^{(m-1)k} 
\end{eqnarray}
Observe that the hypotheses made on $D$ ensure that this equation has degree at
least $2$ in $\Tplanar$.
Moreover, $\Tplanar$ has a positive radius of convergence
$\zc$, and letting $\Tplanarc =\Tplanar(\zc)
$, one has:
\begin{eqnarray}
\label{eq:Rc}
\Tplanarc = (m-1) \sum_{k\in D} k {mk-1 \choose k}  
[\zc \Tplanarc^{m-1}]^{k}
\end{eqnarray}
Subtracting Equation~\ref{eq:planar} to Equation~\ref{eq:Rc} shows that 
$\zc \Tplanarc^{m-1}=\tc$, and 
$\Tplanarc=\beta_{m,D}$, where $\tc$ and $\beta_{m,D}$
are defined in the statement of Theorem~\ref{thm:main}.

Writing down the multivariate Taylor expansion of Equation \ref{eq:planar} near
$z=\zc$ easily leads to the following lemma:
\begin{lemma}
\label{lemma:devR}
When $z$ tends to $\zc$, the following Puiseux expansion holds:
\begin{eqnarray*}
1- \frac{\Tplanar(z)}{\Tplanarc} = \sqrt{\frac{2
\beta_{m,D}}{(m-1)\gamma_{m,D}}} \sqrt{1-\frac{z}{\zc}} +
O\left(\zc-z\right)
\end{eqnarray*}
\end{lemma}

\subsection{The characteristic polynomial of type $0$.}
Let $\mathcal{F}_{m,D}^{\circ\circ}$  be the set of all cells of type $0$ whose
total degree belongs to $mD$. For $F\in\mathcal{F}_{m,D}^{\circ\circ}$, we
denote respectively $|F|$ and $\mathrm{i}(F)$ the size and the increment of
$F$.
The \emph{characteristic polynomial} of type $0$ is the following generating
Laurent polynomial:
\begin{eqnarray*}
P_{m,D}(X,t) = \sum_{F\in \mathcal{F}_{m,D}^{\circ\circ}}
t^{|F|}X^{i(F)}
\end{eqnarray*}
For example, in the case $m=2$, $D=\{2\}$, Figure~\ref{fig:3typesstars} shows
that the characteristic polynomial is $t^2(X^{-1}+1+X)$.
\begin{figure}[h]
\centerline{\includegraphics[scale=1]{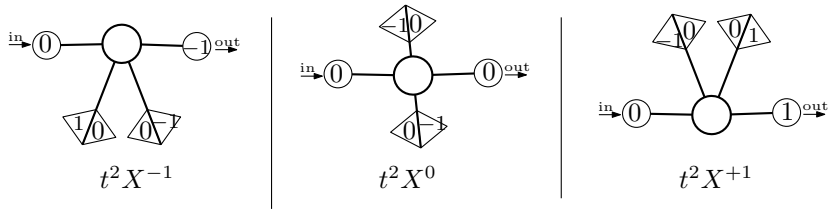}}
\caption{The three cells of type $0$ and total degree $4$.}
\label{fig:3typesstars}
\end{figure}

For every $n\in \mathbb{N}$ and $i\in \mathbb{Z}$, we let $a_{n,i}$ be the
number of chains of type $0$ of total size $n$ and increment $i$. Note that for
every $n$, $a_{n,i}=0$ except for a finite number of values of $i$.
Hence, if $\mathbb{C}[X,X^{-1}][[t]]$ denotes the ring of formal power series
in $t$ with coefficients that are Laurent polynomials in $X$, the 
generating function $S_{m,D}(X,t) = \sum_{n=0}^\infty \sum_{i=-\infty}^\infty
a_{n,i} t^n X^{i}$ of chains of type $0$ by the size and the increment
is a well defined element of $\mathbb{C}[X,X^{-1}][[t]]$.
%$$
%S_{m,D}(X,t) = \sum_{k=0}^\infty \sum_{i=-\infty}^\infty a_{k,i} t^k X_{i}
%$$
Since, by definition, a chain of type $0$ is a sequence of cells of type $0$,
and since the size and the increment are additive parameters, we have by
classical symbolic combinatorics: $$ S_{m,D}(X,t)= \frac{1}{1-P_{m,D}(X,t)}.
$$ 
This is the reason why we will spent some time on the study of the polynomial
$1-P_{m,D}(X,t)$ (which is called the \emph{kernel} in the standard terminology
of lattice walks, see \cite{BaFl,MBM-Excursions}). 

Observe that, in the $m$-walk reformulation, a cell of type $0$ is a
circular $m$-walk with two distinguished steps $-1$, or equivalently, a 
$m$-walk beginning with a step $-1$, with another step $-1$ distinguished.
Hence the number of cells ot type $0$ and total degree $mk$ equals
$[(m-1)k-1]{mk-1 \choose k}t^k$, so that
$P(1,t)=\sum_{k\in D}[(m-1)k-1]{mk-1 \choose k}t^k$, and
$P(1,\tc)=1$. Consequently, $\tc$ is the radius of convergence of the series
$S_{m,D}(1,t)$.
We now study the partial derivatives at the critical point. We have:
\begin{lemma}
\label{lemma:partialdiff}
\begin{eqnarray}
\frac{t\partial P_{m,D}}{\partial t} (1,\tc) &=&
\frac{\gamma_{m,D}}{m-1} \\ \frac{\partial P_{m,D}}{\partial X}
(1,\tc) &=& 0 \\ \frac{\partial^2 P_{m,D}}{\partial X^2} (1,\tc) &=&
\frac{m}{6}\gamma_{m,D}
\end{eqnarray} 
\end{lemma}

\begin{proof}
The first alinea comes immediately from the definition of $\gamma_{m,D}$ and
the fact that there are $[(m-1)k-1]{mk-1 \choose k}$ distinct cells of type $0$
and size $k$.

For the second alinea, observe that since the operation consisting in inverting the in and out vertices 
of a cell is an involution of $\setcells$, then for every $t$
on has: $P_{m,D}(X,t)=P_{m,D}(X^{-1},t)$, which implies the second claim after
derivating.

We now prove the third alinea. 
%First, recall that $P_{m,D} = \sum_{F\in\mathcal{F}_{m,D}^{\circ\circ}}
%t^{|F|} X^{i(F)}$.
First, recall that in the $m$-walk reformulation, $P_{m,D}$ is the generating
function of linear $m$-walks of length $mk$, beginning with a step $-1$, and
where a position preceding a step $-1$ is distinguished. Since the first derivative 
vanishes (alinea 2), we have:
\begin{eqnarray*}
\frac{\partial^2 P_{m,D}}{\partial X ^2}(1,t)&=&
 \sum_{F\in\setcells} i(F) (i(F)-1) t^{|F|} \\
&=& \sum_{F\in\setcells} i(F)^2 t^{|F|} \;=\; \sum_{k\in D
} p_k t^k
\end{eqnarray*}
where
$\displaystyle
p_k(t) = \sum_{F \in \mathcal{F}_{m,\{k\}}^{\circ\circ}} i(F)^2  
$.
We now fix $k\in D$, and we let $\mathcal{W}_{m,k}^\circ$ be the set of 
$m$-walks of length $mk$ beginning with a step $-1$. We let $u=(m-1)k$
be the number of step $-1$ of such a walk, and for each
$w\in \mathcal{W}_{m,k}^\circ$, we let $x_0(w),x_1(w), \ldots, x_{u-1}(w)$ be
the ordinates of the points preceding a step $-1$ in $w$ (so that $x_0(w)=0$).
Choosing first the $m$-walk, and then distinguishing a step  $-1$, we can write:
%
%A birooted polygon is a rooted polygon where one of the $(m-1)k-1$ 
%active positions has been distinguished. In the rest of the proof we note
%$u=(m-1)k$. Summing over rooted polygons only, we have:
\begin{eqnarray}
\label{eq:TXX1}
p_k&=&
  \sum_{w \in \mathcal{W}_{m,k}^{\circ}} (x_1(w)^2+x_2(w)^2+\dots
  x_{u-1}(w)^2)
\end{eqnarray}
We now introduce the \emph{risings} as the quantities
$\lambda_i(w)=x_{i}(w)-x_{i-1}(w)$, for $i\in\llbracket 1,u\rrbracket$.
Then we have the following facts:
\begin{itemize}
\item[1.] By symmetry, the two following quantities are independant of $j$:
$$V_k=\sum_{w \in\mathcal{W}_{m,k}^{\circ}} \lambda_j(w)^2 t^{k} 
\quad (\mbox{for } j=1..u-1)$$ $$W_k=\sum_{w \in
\mathcal{W}_{m,k}^{\circ}} \lambda_1(w)\lambda_j(w) t^{k(F)} \quad
(\mbox{for } j=2..u-1)$$
\item[2.] Since we have for all $w$: $\delta_1(\delta_1 + \dots + \delta_{u} )
=0$ then it is still true after sumation and:
$$
V_k+[u-1]W_k=0
$$
\end{itemize} 
Putting the last fact together with Equation~\ref{eq:TXX1}, one gets after 
replacing $x_i(w)$ by $\delta_1(w)+\dots+\delta_i(w)$ and expanding:
\begin{eqnarray}
p_k=&=&
\frac{u(u-1)}{2} V_k + \frac{u(u-1)(u-2)}{3} W_k \nonumber \\
&=& \frac{u(u+1)}{6} V_k
\label{eq:tXX2}
\end{eqnarray}
Now, for any integer $i$, the number of rooted polygons of size $k$
such that $\delta_1=(m-1)i-1$ is easily seen to be equal to ${mk-2-i\choose k-i}$, hence:
\begin{eqnarray*}
V_k&=& \sum_i [(m-1)i-1]^2 {mk-2-i\choose k-i}  \\
&=& [Y^{k-i}] \sum_i [(m-1)i-1]^2  (1+Y)^{mk-2-i}  
\end{eqnarray*}
Expressing the last sum as an explicit rational fraction in $Y$, one obtains the
exact value of $V_k$, and putting it together with Equation~\ref{eq:tXX2} leads
to: $$
p_k = \frac{mk(m-1)[(m-1)k-1]}{6}{mk-1 \choose k}
$$
Hence
$$ 
\frac{\partial^2 P_{m,D}}{\partial X ^2}(1,t)=
\sum_{k\in D} \frac{mk(m-1)[(m-1)k-1]}{6}{mk-1 \choose k} t^{k}
$$
which together with the definition of $\gamma_{m,D}$ concludes the proof of the
lemma.
\end{proof}

\subsection{The roots of the characteristic polynomial.}

In this section, we study the roots of $1-P_{m,D}$. Some of the arguments are
general for lattice walks and already contained in
\cite{BaFl},\cite{MBM-Excursions}.

Let $r=(m-1)\max(D) -1$ be the maximal increment of a cell of type $0$ with
total degree in $mD$. Then $P_{m,D}$ has positive degree $r$, and negative degree $-r$.

Among the $2r$ roots of $1-P_{m,D}$, exactly $r$ are finite at $t=0$. Call them
$\alpha_1(t),\dots,\alpha_r(t)$. Since inverting the in and out vertices is
an involution of the set of cells of type $0$, $P_{m,D}$ is symmetric under
the exchange $X \leftrightarrow X^{-1}$, and the $r$ other roots are 
$\alpha_1^{-1}, \dots, \alpha_r^{-1}$, and they are infinite at $t=0$.
%Moreover, one can check that the $alpha_i$ are the $r$ determinations of a Puiseux
%series at $t=0$, the first term being $t^{1/r}$.
% We denote $\alpha_1(t)$, and call
%the \emph{principal branch}, the root which corresponds to the standard determination of
%the $r$-th root.
We have more precisely:
\begin{lemma}
\label{lemma:roots}
Up to renumbering the roots, we have:
\begin{itemize}
\item[(i)]
 $\alpha_1 \in \mathbb{R}$ for $t\in [0,\tc]$, and $\alpha_1(t)$ is an
 increasing function on this interval. Moreover, 
 $\alpha_1 \longrightarrow 1$ when $t\longrightarrow \tc$.
\item[(ii)]
 for all $i \neq 1$, and for all $t\in[0,\tc]$, 
 $|\alpha_i(t)|<|\alpha_1(t)|$. There exists $\epsilon>0$ such
 that for all $i\neq 1$ and for all $t\in [0,t_c]$, 
 $|\alpha_i(t)|<1-\epsilon$.
\end{itemize}
\end{lemma}
In the rest of the paper, we will keep the renumbering of the roots given by the lemma.
The root $\alpha_1(t)$ is called the \emph{principal branch}.
\begin{proof}
We already observed that $1$ is a root of $P_{m,D}(X,t_c)$, and by
Lemma~\ref{lemma:partialdiff}, it is of multiplicity exactly two.
%So there exists an $i$ such that $\alpha_i(t_c)=1$, and we will assume that $i=1$.
%Now, since $P_D$ is symmetrical in $X$ and $X^{-1}$, and since it has positive 
%coefficients, a simple computation now shows that :
%$$\frac{\partial P_D}{\partial t} (1,t_c)=0
%\mbox{ and }
%\frac{\partial^2 P_D}{\partial t^2} (1,t_c)>0
%$$
%which prove that $1$ is root of multiplicity exactly two of $1-P_D$.

Now, for every $t\in(0,t_c)$ one has by positivity of the
coefficients $P_{m,D}(1,t)\leq P_{m,D}(1,t_c)=1$, and $P_{m,D}(0,t)=\infty$.
Moreover $X \mapsto P_{m,D}(X,t)$ is a decreasing function on $[0,1]$ 
(since for all $i$, $X^i+X^{-i}$ is) so that there exists
an unique $\alpha=\alpha(t) \in [0,1]$ such that $P_{m,D}(\alpha)=1$. 
%Moreover, for the
%same reason, 
Now, since $P_{m,D}$ has positive coefficients, $\alpha(t)$ is an increasing
function of $t$. This proves claim (i).

Now, for every $\lambda \in \mathbb{C}$,
one has $|P_{m,D}(\lambda)|\leq P_{m,D}(|\lambda|)$, with equality if and only
if $\lambda>0$. Hence if $|\lambda|\leq \alpha_1(t)$ one has
$P_{m,D}(\lambda,t)\leq 1$, with equality if and only if $\lambda=\alpha_1$. This, together with a compacity argument,  
implies claim (ii).
\end{proof}
Lemma~\ref{lemma:partialdiff} then gives:
\begin{lemma}
\label{lemma:devalpha1}
The following Puiseux expansion holds near $t=t_c$:
$$
\alpha_1(t) = 1 - \sqrt{\frac{12}{m(m-1)}}
\left(1-\frac{t}{\tc}\right)^{\frac{1}{2}} +
o\left(\left(1-\frac{t}{\tc}\right)^{\frac{1}{2}}\right) $$
\end{lemma}
Let us now define, for $i\in\llbracket 1,r \rrbracket$, the following power
series in $t$:
\begin{eqnarray}
\label{eq:defCi}
C_i = \frac{1}{ t^{2r}\alpha_i\prod_{j} (1-\frac{1}{\alpha_i\alpha_j})
 \prod_{j\neq i} (\alpha_i-\alpha_j)}
\end{eqnarray}
Observe that $C_i$ is a well defined element of $C[[t]]$.
Moreover, the following partial fraction expansion holds:
\begin{eqnarray}
S_{m,D}(X,t)&=& \frac{1}{\prod_i (X-\alpha_i)(1-{\alpha_i}^{-1} X^{-1})}
\label{eq:parfrac}\\ &=& \sum_i \frac{C_i\alpha_i}{X-\alpha_i} + \sum_i
      \frac{C_i}{1-X\alpha_i}\label{eq:Sgeom} 
\end{eqnarray}
For all $n\in\mathbb{Z}$ we let $M_n(t)=\sum_{k=0}^\infty a_{k,n} t^k$ be the generating series of chains of type $0$ of total
increment $n$, by the size. Now, it easy to extract the coefficient of $X^n$ in Equation~\ref{eq:Sgeom},
via the following manipulations
\footnote{The author knows this trick from Mireille Bousquet-M\'elou.}
 in the ring $\mathbb{C}[X,X^{-1}][[t]]$:
\begin{eqnarray*}
S_{m,D}(X,t)&=& \sum_i \frac{X^{-1}C_i\alpha_i}{1-X^{-1}\alpha_i} + \sum_i \frac{C_i}{1-X\alpha_i}\\
&=& \sum_i \sum_{n=0}^{\infty} C_i {\alpha_i}^{n+1} X^{-n-1}
   -\sum_i \sum_{n=0}^{\infty} C_i {\alpha_i}^{n} X^{n}
\end{eqnarray*}
so that one obtains the generating function of chains of increment $n\in
\mathbb{Z}$:
\begin{eqnarray}
\label{eq:walks}
M_n(t)&=& "[X^{n}]S_{m,D}(X,t)" = \sum_i C_i(t) \alpha_i(t)^{|n|}
\end{eqnarray}
Observe that in the series $M_0(t)$, the empty walk of length $0$ is counted.

\subsection{Chains of all types.}
We will see now that the generating series of chains of type $0$, and of type
$\tau\neq 0$ are closely related. To put this relation in a more fancy form, we
consider not only chains, but chains where a planar mobile has been attached to
each labelled vertex.
For all $\tau\in\llbracket 0 ,m-1 \rrbracket$, we let 
$H^\tau_n(z)$ be the generating series of chains of type $\tau$, that carry on
each labelled corner a planar mobile (which can eventually be trivial). The
variable $z$ counts the total number of flagged edges.
%
%=\sum_{\mathfrak{c}} z^{|\mathfrak{c}|}$, where the sum is
%taken over all the chains of type $\tau$ and total increment $n$, where one
%eventually trivial planar mobile has been attached to each labelled corner.

In the case $\tau=0$, this series is easily related to $M_n$:
since a chain of type $0$ and size $k$ has $(m-1)k$ labelled vertices, and $k$
flagged edges, $H_n$ is obtained from $M_n$ by the substitution $z\leftarrow
z\Tplanar(z)^{m-1}(z)$.
\begin{definition}
In the rest of the paper, we note $t(z):=z\Tplanar(z)^{m-1}$
\end{definition}
We have then: $$
H^0_n(z) = M_n\left(t(z)\right) = 
\sum_i C_i(t(z)) \alpha_i(t(z))^{|n|}
$$

We now examine the case $\tau\in\llbracket 1, m-1 \rrbracket$. For such $\tau$,
we let $
P_{m,D}^\tau(X,t,u)
$
be the generating polynomial of elementary cells of type $\tau$, where $X$,
$t$, $u$ count respectively the increment, the number of black vertices, and
the number of labelled vertices.
We also let $r^{(\tau)}_k(X)$ be the generating series of elementary white stars
of total degree $mk$,  with exactly two special split-edges, one of type $\tau$ and one
of type $m-\tau$. Here, the variable $X$ counts the increment between the two special edges.
Since such a star has exactly $(m-1)k-1$ labelled vertices, $k-1$ black
vertices, and since the generating series of black stars of degree $m$ with two
special edges is $1+X+\dots+X^{m-2}$, one has, recalling that a cell of type
$\tau$ is the juxtaposition of a white and a black star: 
\begin{eqnarray}
\label{eq:PDtau}
P_{m,D}^{(\tau)}(X,t,u) =
(1+X+\dots+X^{m-2})\sum_{k\in D}t^{k} u^{(m-1)k-1}r^{(\tau)}_k(X).
\end{eqnarray}
\begin{figure}[h]
\centerline{\includegraphics[scale=1.0]{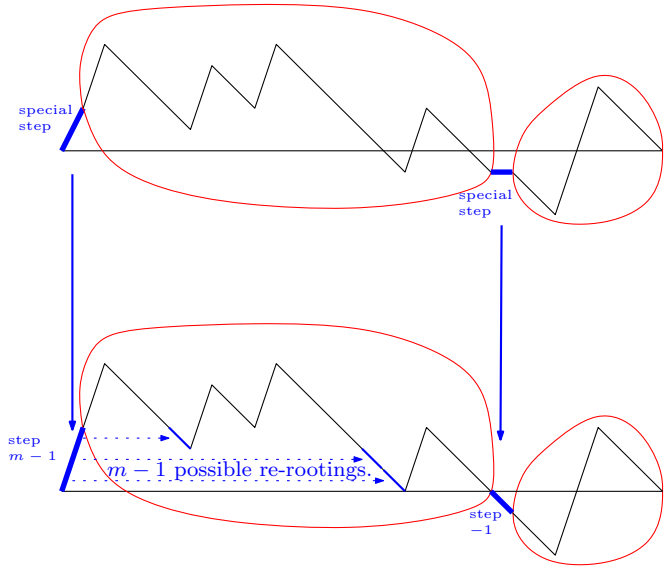}}
\caption{Walks with two special steps are in correspondance with walks with two
distinguished steps, of increments $m-1$ and $-1$ (vertical arrows). These walks
can be re-rooted in $m-1$ different ways to obtain walks with two distinguished steps $-1$ (horizontal arrows). 
In the re-rooting operation, the increment between the two steps is modified by a quantity among 
$0,\dots,m-2$, inducing a factor $1+X+\dots+X^{m-2}$ in the generating
series.}
\label{fig:PDtau}
\end{figure}
Now, by Lemma~\ref{lemma:mwalk}, $r^{(\tau)}_k(X)$ is also the generating series
of walks of length $mk$, with $k-1$ steps $m-1$, $(m-1)k-1$ steps $-1$, beginning with a step $\tau-1$
and ending by a step $m-\tau-1$. These walks are in bijection with walks with
of length $km$ with only steps $-1$ and $m-1$, beginning with a step $m-1$, and
with a distinguished step $-1$: to see that, exchange the steps $\tau$, $m-2-\tau$ by two steps $-1,m-1$
\footnote{In particular, $r^{(\tau)}_k(X)$ does not depend on $\tau$.}. 
Since in that walk the only decreasing steps are steps $-1$,
the distinguished step $m-1$ lies in front of exactly $m-1$ steps $-1$.
Hence (see Figure~\ref{fig:PDtau})
$
(1+X+\dots+X^{m-2})r_k(X)
$
is the generating series of walks with two distinguished steps $-1$, 
where $X$ counts the increment between them. 

Observe that these two distinguished steps are
not necessarily distinct. If they are equal, we have a circular walk with one
marked step $-1$: there are ${mk-1 \choose k}$ of those.
If they are not equal, the object considered is, up to the correspondance of 
Lemma~\ref{lemma:mwalk}, a cell of type $0$. Hence we have:
$$
(1+X+\dots+X^{m-2})r_k(X)
 = 
{mk-1 \choose k} 
+ [t^k]P_{m,D}(X,t)
$$
This gives with Equation~\ref{eq:PDtau}:
$$
\Tplanar(z) P_{m,D}^\tau(X,z,\Tplanar(z)) = 
\left(\sum_{k\in D}{mk-1 \choose k}t(z)^k + P_{m,D}(X,t(z)) \right)
$$
And using Equation~\ref{eq:planar} gives:
$$
\frac{\Tplanar(z)}{1-P_{m,D}(X,t(z))} =
\frac{1}{1-P_{m,D}^\tau(X,z,\Tplanar(z))}$$
Now, observe that the coefficient of $X^n$ is the right-hand side is precisely
the series $H_n^\tau(z)$. On the other hand, the coefficient of $X^n$ in the
left-hand side equals $\Tplanar M_n(t(z))$. This gives 
the following proposition, which is the key that relates
the enumeration of $m$-hypermaps and $m$-constellations:
\begin{proposition}
For all $\tau\in\llbracket 1 ,m-1 \rrbracket$, and for all $n\in\mathbb{Z}$, we
have:
\begin{eqnarray}
\label{eq:Htau}
H_n^\tau(z) = \Tplanar(z) H_n^0(z)
\end{eqnarray} 
\end{proposition}

\section{Generating series of mobiles}

\subsection{Translating Proposition~\ref{prop:algo} into generating series.}
The previous section gave us all the building blocks to translate
Proposition~\ref{prop:algo} in terms of generating series.

Let $(\mathfrak{s},\tau,F,\mathfrak{a},\lambda)$ be a full scheme of genus $g$.
We are going to use Algorithm~\ref{algo:reconstruct}, and substitute each edge
of $\mathfrak{s}$ with a chain. 
We first choose a compatible labelling $(l_v)_{v\in V(\mathfrak{s})}$ of that
scheme. We need a little discussion on a
special case. Imagine that the labelling imposes to substitute an edge
$e$ of type $0$ to a chain of type $0$ of increment $\Delta(e)=0$. Then, if one of the extremities of
$e$ is associated with a non trivial  nodal star, it is possible to substitute
$e$ to an empty chain ; otherwise, if the two extremities are associated with
the trivial nodal star $\circ$, the chain of length $0$ is excluded: this would
identify the two vertices of the chain.
Hence, if $e$ is an edge of $\mathfrak{s}$, of
extremities $v_1$ and $v_2$, we set: $$
r(e) = \left\{\begin{array}{l} 1 \mbox{ if }
\tau(e)=0 \mbox{ and }
\Delta(e)=0 \mbox{ and }
F_{v_1}=F_{v_2}=\circ\\
0 \mbox{ otherwise.}
\end{array}
\right.
$$
Then the edge $e$ can be replaced by the empty walk if and only if
$r(e)\neq 1$. Observe that $r(e)$ depends actually only on the full scheme, but
not on the compatible labelling $(l_v)$ itself.

We let $|\mathfrak{a}|=|\mathfrak{a}_1|+|\mathfrak{a}_2|$,
$\langle\mathfrak{a}\rangle=\langle\mathfrak{a}_1\rangle+\langle\mathfrak{a}_2\rangle$, 
and similarly $|F|=\sum_v |F_v|$ and $\langle F\rangle=\sum_v \langle
F_v\rangle$.
Hence the series:
$$
R_{\mathfrak{s},\tau,F,\mathfrak{a},\lambda}(z):=
z^{\langle \mathfrak{a} \rangle+\langle \mathfrak{c} \rangle}
t(z)^{| \mathfrak{a} |+|\mathfrak{c} |}
\sum_{\scriptsize labellings}
\prod_{e\in E(\mathfrak{s})}
\left(H_{\Delta(e)}^{\tau(e)}(z)-r(e)\right)
$$
is the generating series of objects generated by the first four steps of
Algorithm~\ref{algo:reconstruct}. Observe the first and second factor, that
accout respectively for the fact that black vertices appearing in the full
scheme must be counted, and that planar mobiles must be attached also on the
labelled vertices of the full scheme. 

We now let $R_g(z)$ be the generating series of all mobiles of genus $g$, by the
number of black vertices. Again, dependency in $m$ and $D$ are omitted in the
notation. Since a mobile with $k$ black vertices has in total $mk$ edges,
step 5 in Algorithm~\ref{algo:reconstruct} corresponds to an operator
$m\frac{zd}{dz}$ on the generating series. Hence, in terms of generating series,
Proposition~\ref{prop:algo} admits the following reformulation:
\begin{corollary}\label{cor:algo}
\begin{eqnarray}\label{eq:Rgsum}
R_g(z) =
m\frac{zd}{dz}
\sum_{(\mathfrak{s},\tau,F,\mathfrak{a},\lambda)\in\mathcal{F}_g}
\frac{1}{2|E(\mathfrak{s})|} 
R_{\mathfrak{s},\tau,F,\mathfrak{a},\lambda}(z)\end{eqnarray}
\end{corollary}
\begin{remark}
It follows from remark~\ref{rem:algobipartite} that the generating series of
mobiles corresponding to $m$-constellations of degree set $mD$ can be written:
\begin{eqnarray}
R^{cons}_g(z) =
m\frac{zd}{dz}
\sum_{(\mathfrak{s},\vec 0,F,\mathfrak{a},\lambda)\in\mathcal{F}_g}
\frac{1}{2|E(\mathfrak{s})|} 
R_{\mathfrak{s},\tau,F,\mathfrak{a},\lambda}(z)\end{eqnarray}
where the sum is restricted to the full schemes
$(\mathfrak{s},\tau,F,\mathfrak{a},\lambda)\in\mathcal{F}_g$ such that $\tau$
associates $0$ to all edges.
\end{remark}

\subsection{An exact computation.}
We fix a full scheme $\mathfrak{f}=(\mathfrak{s},\tau,F,\mathfrak{a},\lambda)$.
We let $E_1$ be the set of edges of $\mathfrak{s}$ such that $r(e)=1$, and $E_2$ be its
complementary.
 
To lighten notations, we note $\Tplanar$, $C_i$, $\alpha_i$ for
$\Tplanar(z)$, $C_i(t(z))$ and $\alpha_i(t(z))$, respectively.
We also note $z^\mathfrak{f}:=z^{\langle \mathfrak{a} \rangle+\langle \mathfrak{c} \rangle}
t(z)^{| \mathfrak{a} |+|\mathfrak{c} |}.
$
Then we have from Equation~\ref{eq:Htau}
\begin{eqnarray}
R_\mathfrak{f} &=& 
z^\mathfrak{f} \sum_{\delta_1,..\delta_M>0}
\prod_{e\in E_1} \left( \sum_{i=1}^rC_i - 1\right) 
\prod_{e\in E_2} \left(
\Tplanar^{\mathbbm{1}_{\tau(e)\neq0}}\sum_{i=1}^rC_i\alpha_i^{|\Delta(e)|}
\right)
\\&=&
z^\mathfrak{f} 
\Tplanar^{n_{\neq}}
\left( \sum_{i=1}^rC_i - 1\right)^{|E_1|} 
\sum_{\delta_1,..\delta_M>0}
\prod_{e\in E_2} \left(
\sum_{i=1}^rC_i\alpha_i^{|a_{F,\mathfrak{a}}(e)+\sum_{j=1}^M A_{e,j}\delta_j |}
\right)\nonumber
\end{eqnarray}
where $n_{\neq}$ is the number of edges of $\mathfrak{s}$ of type $\neq 0$.
Now we have by expanding the product:
\begin{eqnarray}
&&\sum_{\delta_1,..\delta_M>0}
\prod_{e\in E_2} \left(
\sum_{i=1}^rC_i\alpha_i^{|a_{F,\mathfrak{a}}(e)+\sum_{j=1}^M A_{e,j}\delta_j |}
\right)\nonumber\\
&=&
\sum_{\delta_1,..\delta_M>0}
\sum_{i\in \llbracket1,r\rrbracket^{E_2}} \prod_{e\in E_2}
C_{i_e} \alpha_{i_e}^{|a_{F,\mathfrak{a}}(e)+\sum_{j=1}^M A_{e,j}\delta_j
|}\nonumber\\
&=&
\sum_{i\in \llbracket1,r\rrbracket^{E_2}} \prod_{e\in E_2}
C_{i_e} 
\sum_{\delta_1,..\delta_M>0}\prod_{e\in E_2}
\alpha_{i_e}^{|a_{F,\mathfrak{a}}(e)+\sum_{j=1}^M A_{e,j}\delta_j
|}\label{eq:interm}
\end{eqnarray}
Now, observe that when the $\delta_i$'s are large enonugh (say $\geq$ some
number $K$), all the quantities $a_{F,\mathfrak{a}}(e)+\sum_{j=1}^M
A_{e,j}\delta_j$ are positive, so that we can remove the absolute value in the
sum above. If we define the polynomial:
$$\mathfrak{p}(\alpha_1,\ldots,\alpha_r):=
\sum_{\delta_1,..\delta_M< K}\prod_{e\in E_2}
\alpha_{i_e}^{|a_{F,\mathfrak{a}}(e)+\sum_{j=1}^M A_{e,j}\delta_j|}
$$
then the quantity~\ref{eq:interm} rewrites:
\begin{eqnarray*}&&
\sum_{i\in \llbracket1,r\rrbracket^{E_2}} \prod_{e\in E_2}
C_{i_e} \left(\mathfrak{p}(\alpha_1,\ldots,\alpha_r)	+
\sum_{\delta_1,..\delta_M>K}\prod_{e\in E_2}
\alpha_{i_e}^{a_{F,\mathfrak{a}}(e)+\sum_{j=1}^M A_{e,j}\delta_j}
\right)\\
&=&
\sum_{i\in \llbracket1,r\rrbracket^{E_2}} \prod_{e\in E_2}
C_{i_e} \left(\mathfrak{p}(\alpha_1,\ldots,\alpha_r)	+
\prod_{e\in E_2} \alpha_{i_e}^{a_{F,\mathfrak{a}}(e)}
\prod_{j=1}^M
\frac{\left(\prod_e \alpha_{i_e}^{A_{e,j}}\right)^K}{1-\prod_e
\alpha_{i_e}^{A_{e,j}}}\right)
\end{eqnarray*}
where passing from the first to the second line is just a geometric summation
on each variable $\delta_j$. Observe that it remains only sums and products
aver finite sets. This gives the statement:
\begin{proposition}
The series $R_\mathfrak{f}(z)$ is an algebraic series of $z$, given by the
following expression:
\begin{eqnarray}\label{eq:Rfalgebraic}
R_\mathfrak{f}(z)=z^\mathfrak{f}\Tplanar^{n_{\neq}}
\left( \sum_{i=1}^rC_i - 1\right)^{|E_1|} 
\sum_{i\in \llbracket1,r\rrbracket^{E_2}} \prod_{e\in E_2}
C_{i_e} \left(\mathfrak{p}(\alpha_1,\ldots,\alpha_r)	+
\prod_{e\in E_2} \alpha_{i_e}^{a_{F,\mathfrak{a}}(e)}
\prod_{j=1}^M
\frac{\left(\prod_e \alpha_{i_e}^{A_{e,j}}\right)^K}{1-\prod_e
\alpha_{i_e}^{A_{e,j}}}\right)
\end{eqnarray}
\end{proposition}
One should not worry to much about the form of the last equation. In the
asymptotic regime, many terms will disappear, and it will look much nicer.

\subsection{The singular behaviour of $R_\mathfrak{f}$.}

\begin{lemma}
The radius of convergence of $R_\mathfrak{f}(z)$ is at least $\zc$.
\end{lemma}
\begin{proof}
Let us consider the family of all objects obtained by replacing each edge $e$ of
the scheme $\mathfrak{s}$ by a chain of type $\tau(e)$, without any constraint
on the increment of the chains. These objects are not all valid mobiles
(most of them are not) but clearly, this family contains all the mobiles
counted by the series $R_\mathfrak{f}(z)$. Now, is $\mathfrak{s}$ has $n_0$
edges of type $0$ and $n_1$ edges of type $\neq0$, the generating series of
these objects is: $$
z^\mathfrak{f} \left(\frac{1}{1-P_{m,D}(1,t(z))}\right)^{n_0}
\left(\frac{\Tplanar(z)}{1-P_{m,D}(1,t(z))}\right)^{n_1}
$$
so that:
$$R_\mathfrak{f}(z) \preccurlyeq
z^\mathfrak{f} \Tplanar(z)^{n_1}
\left(\frac{1}{1-P_{m,D}(1,t(z))}\right)^{n_0+n_1}
$$
where $\sum f_n z^n \preccurlyeq \sum g_n z^n$ means that $f_n \leq g_n$ for
all $n$. Since all the coefficients of these two series are nonnegative, this implies
that the radius of convergence of $R_{\vec\mathfrak{o}}$ is at least $\zc$
(recall that $P_{m,D}(1,\tc)=1$ and that $P_{m,D}$ has positive coefficients, so that
$\zc$ is indeed the radius of convergence of the right hand side).
\end{proof}

We now study the behaviour of $R_\mathfrak{f}(z)$ near $z=\zc$. Several things
happen that create a singularity: First, $\zc$ is the radius of convergence of
$\Tplanar$ and $t(z)$. Second, we saw that at $t=\tc$, at least $\alpha_1(t)$
ceases to be analytic: we are thus in a regime of \emph{composition of
singularities}. Third, at $t=\tc$, $\alpha_1(\tc)=1$ so that denominators in
Equation~\ref{eq:Rfalgebraic} can vanish. These three factors are easy to
control. There is a last one, however, that could happen. Indeed, if
$P_{m,D}(X,\tc)$ has other multiple roots than $1$, the corresponding series $C_i$
diverge. However, if ever this happens
%\footnote{
%It is conjectured in the thesis of Cyril Banderier that this never
%happens, even if one replaces $P_{m,D}$  by any polynomial with positive
%coefficients.} 
the corresponding divergences will cancel between multiple
roots, and \emph{everything works as if $1$ was the only multiple root}.
Precisely, we have:
\begin{proposition}
The only dominating term in
Expression~\ref{eq:Rfalgebraic} is the one
corresponding to $i_e=1$ for all $e$, and when $z$ tends to $\zc$ we have:
\begin{eqnarray}
R_\mathfrak{f}(z)
&=& c_{\mathfrak{s},\lambda} z^{\mathfrak{f}} \Tplanarc^{n_{\neq}}
\frac{C_1(t(z))^{|E(\mathfrak{s})|}} {\left[1-\alpha_1(t(z))\right]^{M}}
\left[1+o(1)\right]
\end{eqnarray}
where the constant
$\displaystyle c_{\mathfrak{s},\lambda}=
\frac{1}{\prod_{j=1}^M \sum_{e\in E} A_{e,j}}
$ depends only on $\mathfrak{s}$ and $\lambda$.
\end{proposition}
\begin{proof}
First, Lemma~\ref{lemma:devalpha1} and the definition of $C_1$ ensures that
when $t$ tends to $\tc$: 
$$
C_1(t)=\Theta\left((\tc-t)^{-1/2}\right).
$$
Moreover, from the definition of $C_i$, and from Lemma~\ref{lemma:roots}, if
$\alpha_i(\tc)$ is a root of order $L$ of $P_{m,D}(X,\tc)$, one has:
$$
C_i(t)=\Theta\left((\tc-t)^{-\frac{L-1}{L}}\right).
$$
which dominates $C_1(t)$ if $L\geq 3$.
We now define the equivalence relation on $\llbracket 1, r \rrbracket$
$i\approx j$ if $\alpha_i(\tc)=\alpha_j(\tc)$, and we consider the
corresponding partition in classes:
$\llbracket 1, r \rrbracket = \cup_{q=1}^l I_q$, where $l$ is the number of
classes. Observe that $1$ is alone in its class, and we assume that $I_1=\{1\}$.
It is easily seen (for example with a Newton-Puiseux
expansion of the $\alpha_i$'s) that for each $q\neq 1$, one has when $t$ tends
to $\tc$: 
\begin{eqnarray}\label{eq:conjugateroots}
\sum_{i\in I_q}C_i(t) = O(1).
\end{eqnarray}
We now partition $\llbracket1,r\rrbracket^{E_2}$ according to which indices
are in which class $I_q$. We have:
\begin{eqnarray}
&&
\sum_{i\in \llbracket1,r\rrbracket^{E_2}} \prod_{e\in E_2}
C_{i_e} \left(\mathfrak{p}(\alpha_1,\ldots,\alpha_r)	+
\prod_{e\in E_2} \alpha_{i_e}^{a_{F,\mathfrak{a}}(e)}
\prod_{j=1}^M
\frac{\left(\prod_e \alpha_{i_e}^{A_{e,j}}\right)^K}{1-\prod_e
\alpha_{i_e}^{A_{e,j}}}\right) \nonumber \\
&=& \sum_{w\in\llbracket1,l\rrbracket^{E_2}}
\sum_{i_1\in I_{w_1}}\ldots \sum_{i_{k}\in I_{w_{k}}}
\prod_{e\in E_2}
C_{i_e} \left(\mathfrak{p}(\alpha_1,\ldots,\alpha_r)	+
\prod_{e\in E_2} \alpha_{i_e}^{a_{F,\mathfrak{a}}(e)}
\prod_{j=1}^M
\frac{\left(\prod_e \alpha_{i_e}^{A_{e,j}}\right)^K}{1-\prod_e
\alpha_{i_e}^{A_{e,j}}}\right) \ \ \ \ \  \label{eq:interm2}
\end{eqnarray}
For each $w\in\llbracket1,l\rrbracket^{E_2}$, we note 
$\mathfrak{p}_w$ the value at $t=\tc$ of the polynomial
$\mathfrak{p}(\alpha_1,\ldots.\alpha_r)$, and we let 
$k_w=\#\{j,\  \forall e\in E_2,\ A_{e,j}=0 \mbox{ or }i_e=1\}$ 
be the number of $j$'s for which the denominator vanishes, so that we
have: $$
\prod_{e\in E_2} \alpha_{i_e}^{a_{F,\mathfrak{a}}(e)}
\prod_{j=1}^M
\frac{\left(\prod_e \alpha_{i_e}^{A_{e,j}}\right)^K}{1-\prod_e
\alpha_{i_e}^{A_{e,j}}}
= \frac{\aleph_w}{(1-\alpha_1(t))^{k_w}} [1+o(1)]
$$ for some quantity $\aleph_w$ that depends only on $w$.
Then the quantity~\ref{eq:interm2} rewrites:
\begin{eqnarray*}
&&  [1+o(1)]
\sum_{w\in\llbracket1,l\rrbracket^{E_2}}
\left(\mathfrak{p}_w + \frac{\aleph_w}{(1-\alpha_1(t))^{k_w}} \right)
\sum_{\i_1\in I_{w_1}}\ldots \sum_{\i_{k}\in I_{w_{k}}}
\prod_{e\in E_2}
C_{i_e} \\
&=& [1+o(1)]
\sum_{w\in\llbracket1,l\rrbracket^{E_2}}
\left(\mathfrak{p}_w + \frac{\aleph_w}{(1-\alpha_1(t))^{k_w}} \right)
\prod_{e\in E_2} \left(\sum_{i\in I_{w_e}}C_{i}\right)
\end{eqnarray*}
Now, from what we said at the beginning of the proof, $\sum_{i\in
I_{w_e}}C_{i}$ is a $O(1)$ if $w_e\neq 1$, and a $\Theta((\tc-t)^{-1/2})$ if
$w_e=1$. Hence the only dominating term in the last equation is
$w=(1,\ldots,1)$, and this gives finally, returning to
Equation~\ref{eq:Rfalgebraic}: $$
R_\mathfrak{f}(z)=[1+o(1)]
 z^\mathfrak{f} \Tplanarc^{n_\neq} C_1(t(z))^{|E_1|+|E_2|} 
 \prod_{j=1}^M \frac{1}{1-\alpha_1(t(z))^{\sum_{e}A_{e,j}}}
$$
which gives the statement of the proposition.
\end{proof}
Observe that from Equation~\ref{eq:conjugateroots} and~\ref{eq:parfrac}:
$$
S_{m,D}(1,t) = \sum_{q=1}^l \sum_{i\in I_q} \frac{(1+\alpha_i)C_i}{1-\alpha_i}
=\frac{2C_1(t)}{1-\alpha_1}[1+o(1)]
$$
Since Lemma~\ref{lemma:devalpha1} gives the singular expansion of
$\alpha_1(t)$, and since the expansion of
$S_{m,D}(1,t)=\frac{1}{1-P_{m,D}(1,t)}$ follows from Lemma~\ref{lemma:partialdiff}, we obtain:
\begin{lemma}\label{lemma:devC1}
When $t$ tends to $\tc$, the following Puiseux expansion holds:
\begin{eqnarray}
\label{eq:devC1}
C_1(t)= \sqrt{\frac{3(m-1)}{m}}\gamma_{m,D}^{-1}
\left(1-\frac{t}{\tc}\right)^{-1/2} +
o\left(\left(1-\frac{t}{\tc}\right)^{-1/2}\right)
\end{eqnarray}
\end{lemma}
Setting $t=t(z)$, the last proposition and
Lemmas~\ref{lemma:devalpha1},\ref{lemma:devC1}, \ref{lemma:devR} finally give:
\begin{lemma}\label{lemma:devRf}
When $z$ tends to $\zc$, the following Puiseux expansion holds:
\begin{eqnarray}\label{eq:devRfz}
R_\mathfrak{f}(z)= 
c_{\mathfrak{s},\lambda}(\zc)^\mathfrak{f} (\Tplanarc)^{n_{\neq}}
(m-1)^{\frac{k+m}{4}}m^{\frac{M-k}{2}}
\gamma_{m,D}^{\frac{M-3k}{4}}
\beta_{m,D}^{-\frac{k+M}{4}}
3^{\frac{k-M}{2}}
2^{\frac{-k-5M}{4}}
\left(1-\frac{z}{\zc}\right)^{-\frac{k+M}{4}}
[1+o(1)]
\end{eqnarray}
where $k$ is the number of edges of $\mathfrak{s}$.
\end{lemma}

\subsection{The dominant pairs.}
From the last lemma, the singular behaviour of the sum~\ref{eq:Rgsum} is
dominated by the full schemes $\mathfrak{f}$ for which the quantity $k+M$ is
maximal. First, to maximize the quantity $k+M$, we can assume that $\lambda$ is
injective, i.e. that $M=|V(\mathfrak{s})|-1$, so that the dominant terms will be
given by schemes such that the quantity $|E(\mathfrak{s})|+|V(\mathfrak{s})|-1$
is maximal. Now, if a scheme $\mathfrak{s}$ of genus $g$ has $n_i$ vertices of
degree $i$ for all $i\geq 3$ we have:
$$
|E(\mathfrak{s})|+|V(\mathfrak{s})|= 
\sum_{i\geq3}\frac{i+2}{2} n_i.
$$
Maximizing this quantity with the constraint of 
Equation~\ref{eq:Eulerscheme} imposes that $\sum_i n_i$ is maximal, and since
$\sum (i-2)n_i$ is fixed, this is realized if and only if
$n_3\neq0$ and $n_i=0$ for $i\neq 3$, i.e. if $\mathfrak{s}$ has only vertices
of degree $3$. From Euler characteristic formula, such a scheme has $6g-3$
edges and $4g-2$ vertices. This leads to:
\begin{definition}
A \emph{dominant pair} of genus $g$ is a pair $(\mathfrak{s},\lambda)$, where
$\mathfrak{s}$ is a rooted scheme of genus $g$ with $6g-3$ edges and $4g-2$
vertices of degree $3$, and $\lambda$ is bijection:
$V(\mathfrak{s})\rightarrow \llbracket 0, 4g-3 \rrbracket$. \\
The set of all dominant pairs of genus $g$ is denoted $\mathcal{P}_g$.
\end{definition}
Hence, only dominant pairs appear at the first order in the sum~\ref{eq:Rgsum}.

\section{The multiplicative contribution of the nodal stars.}
Observe that Equation~\ref{eq:devRfz} has a remarquable multiplicative form:
the contribution of the pair $(\mathfrak{s},\lambda)$ is clearly separated from
the one of  $(\tau,F,\mathfrak{a})$.
In this section, we will perform a
summation on $(F,\mathfrak{a})$. Since we are only interested in the
asymptotics, we consider only the case of dominant pairs.

\subsection{Four types of nodes}
We fix a triple $(\mathfrak{s},\lambda,\tau)$ such that
$(\mathfrak{s},\tau)$ is a typed scheme and
$(\mathfrak{s},\lambda)\in\mathcal{P}_g$. 

We say that an edge $e\in E(\mathfrak{s})$ is \emph{special} if $\tau(e)\neq0$.
Let $v\in V(\mathfrak{s})$ be a vertex of $\mathfrak{s}$ adjacent to $l$
special edges, and let $\tau_1,..\tau_l$ be their types. We let
$\tilde\tau_i=\tau_i$ if the corresponding edge is incoming at $v$, and $\tilde\tau_i=m-\tau_i$ if it is
outgoing. Hence, from the discussion of subsection~\ref{subsec:Kirchoff},
in any full scheme of the form $(\mathfrak{s},\tau,F,\mathfrak{a},\lambda)$, 
$\tilde\tau_i$ is the type of the corresponding split-edge of $F_v$ if $F_v$ is
a white elementary star ; if $F_v$ is a black elementary star, the corresponding type will be $m-\tilde\tau_i$.
 We have:
\begin{lemma}
The vertices of $\mathfrak{s}$ can be  of four types:
\begin{itemize}
\item[\bf 1.] 
 vertices such that none of the  three adjacent edges are special.
\item[\bf 2.]
 vertices such that exactly two adjacent edges are specials. In this case,
 one has: $\tilde\tau_1 + \tilde\tau_2 = m$
\item[\bf 3.1.]
 vertices such that exactly three edges are specials, and such such that:
  $\tilde\tau_1 + \tilde\tau_2 +\tilde\tau_3 = m$.
\item[\bf 3.2.]
 vertices such that exactly three edges are specials, and such that:
  $\tilde\tau_1 + \tilde\tau_2 +\tilde\tau_3 = 2m$
\end{itemize}
\end{lemma}
\begin{proof}
The lemma is a straightforward consequence of the Kirchoff law
(Proposition~\ref{prop:Kirchoff}), and the fact that the $\tilde\tau_i$'s  are
elements of $\llbracket 1, m-1 \rrbracket$.
\end{proof}
Observe that, in a full scheme, vertices of type {\bf 3.2} can correspond
either to black or white elementary stars, whereas all the other correspond to white
elementary stars only.
We denote by $v_1$ (resp. $v_2$,
$v_3^{(1)}$, $v_3^{(2)}$) the number of vertices of type {\bf 1} (resp. {\bf 2}, {\bf 3.1}, {\bf 3.2}).
Then we have:
\begin{lemma}
$$
v_3^{(1)}=v_3^{(2)}
$$
\end{lemma}
\begin{proof}
Recall that $n_{\neq}$ is the number of edges of type $\neq 0$. 
Counting half-edges implies: $$
2 n_{\neq} = 3 v_3^{(1)}+ 3 v_3^{(2)} + 2 v_2
$$
Now, we compute the total sum, over all edges of type $\neq 0$,
of the quantity $\tau+ (m-\tau)$. It is of course equal to
$mn_{\neq}$, but also to the total sum of the types of the special half-edges leaving all the vertices,
i.e.: $$
m v_3^{(1)} + 2 m v_3^{(2)} + m v_2.
$$
So we have :
$$
\left\{\begin{array}{rcl}
2 n_{\neq} &=& 3 v_3^{(1)}+ 3 v_3^{(2)} + 2 v_2 \\ 
m  n_{\neq} &=& m v_3^{(1)} + 2m v_3^{(2)} + mv_2
\end{array}
\right.
$$
and eliminating $n_{\neq}$ implies the lemma.
\end{proof}

%\subsection{the total contribution}

We let $D_{\mathfrak{s},\lambda,\tau}$ be the set of all pairs
$(F,\mathfrak{a})$ such that
$(\mathfrak{s},\tau,F,\mathfrak{a},\lambda)\in\mathcal{F}_g$. We say that such
a pair is a \emph{decoration} of $\mathfrak{s},\tau,\lambda$. We let $$
R_{\mathfrak{s},\tau,\lambda}(z) = 
\sum_{(F,\mathfrak{a})\in D_{\mathfrak{s},\lambda,\tau}}
R_{\mathfrak{s},\tau,F,\mathfrak{a},\lambda}(z).
$$
Due to the nature of Equation~\ref{eq:devRfz}, we need to compute the sum:
\begin{eqnarray}\label{eq:totalfactor}
\sum_{(F,\mathfrak{a})\in D_{\mathfrak{s},\lambda,\tau}}
\zc^{(\mathfrak{s},\tau,F,\mathfrak{a},\lambda)}.
\end{eqnarray}
Each vertex of $\mathfrak{s}$ will contribute a certain multiplicative
factor to this quantity. 
%Precisely, for each $v$ in $V(\mathfrak{s})$ and 
%$(F,\mathfrak{a})\in D_{\mathfrak{s},\lambda,\tau}$, we let 

\subsubsection{vertices of type 1.}
A vertex $v$ of type one is ajacent to three edges of type $0$. Hence the star
$F_v$ can be either a single vertex $\circ$, either an elementary white star
with three distinguished labelled vertices. The corresponding multiplicative
factor is therefore:
$$
1+\sum_{k\in D} \frac{[(m-1)k][(m-1)k-1]}{2}{mk-1 \choose k} \tc^{(m-1)k} =
\frac{\gamma_{m,D}}{2}.$$
Moreover, in this case, the half-edges ajacent to $e$ are all of type $0$,
so they do not carry any correcting star of $\mathfrak{a}$.

\subsubsection{vertices of type 2.}
First, a vertex of type two cannot be decorated by a black star, since it is
linked to an edge of type $0$. Then, a vertex of type $2$ corresponds to a white
elementary star with exactly two special edges, which is rooted at a
labelled vertex. There are $\frac{k[(m-1)k-1]}{2}{mk-1 \choose k}$ of those.
Moreover, each time a special half-edge is outgoing at $v$, we need to add a
correction black star in $\mathfrak{a}$ for the corresponding superchain to
begin with a white star. Observe that that the number of black stars with two
marked special edges is $(m-1)$, so that each black star added in
$\mathfrak{a}$ contributes a factor $(m-1)\zc$ at the critical point.
Hence the multiplicative contribution of a vertex of type {\bf 2} is: \begin{eqnarray*}&&
[\zc(m-1)]^{out(v)} 
\sum_{k\in D}\frac{k[(m-1)k-1]}{2}{mk-1 \choose k}
\zc^{k}\Tplanarc^{(m-1)k-1}\\ &=& 
\frac{[\zc(m-1)]^{out(v)}}{\Tplanarc} \frac{\gamma_{m,D}}{2}
\end{eqnarray*}
where we noted $out(v)$ the number of outgoing special half-edges at $v$.

\subsubsection{vertices of type 3.1}
Such a vertex can correspond only to a white star. In the Motzkin walk
reformulation, this star is a walk of length $mk\in mD$, with $(m-1)k-2$
steps $-1$, $k-1$ steps $m-1$, that begins with a special step, and with two
other special steps. For a given $k$, the number of such walks is
${mk-1 \choose(m-1)k-2,k-1,2}=\frac{k [(m-1)k-1] }{2} {mk-1 \choose k } $.
Moreover, as before, for each outgoing edge, we have to add a black polygon in
the sequence $\mathfrak{a}$, so that the multiplicative contribution of a vertex of type {\bf 3.1} is
finally:
\begin{eqnarray*}&&
[\zc(m-1)]^{out(v)} 
\sum_{k\in D}\frac{k[(m-1)k-1]}{2}{mk-1 \choose k}
\zc^{k-1}\Tplanarc^{(m-1)k-2}\\ &=& 
\frac{[\zc(m-1)]^{out(v)}}{(m-1)\zc\Tplanarc^2} \frac{\gamma_{m,D}}{2}
=
{[\zc(m-1)]^{out(v)-1}} \frac{\gamma_{m,D}}{2\Tplanarc^2}
\end{eqnarray*}

\subsubsection{vertices of type 3.2}
Such a vertex can correspond to a white or black star.

If it is decorated by a white star, it corresponds to a walk of length $mk\in
mD$, with $(m-1)k-1$ steps $-1$, $k-2$ steps $m-1$, beginning with a special
step, and with two other special steps. The number of such walks being 
${mk-1 \choose(m-1)k-1,k-2,2}=\frac{k [k-1] }{2} {mk-1 \choose k } $,
the corresponding contribution
is:
\begin{eqnarray*}&&
[\zc(m-1)]^{out(v)} 
\sum_{k\in D}\frac{k [k-1]}{2}{mk-1 \choose k}
\zc^{k-2}\Tplanarc^{(m-1)k-1}\\ &=& 
[\zc(m-1)]^{out(v)}\left(
\frac{1}{\zc^2\Tplanarc} \frac{\gamma_{m,D}-(m-2)\beta_{m,D}}{2(m-1)^2}
\right) 
\end{eqnarray*}

In the other case, $v$ is decorated by a black star with three marked special
edges: there are $\frac{(m-1)(m-2)}{2}$ of those, so that the contribution of
the black star is $\frac{(m-1)(m-2)}{2}\zc$.
Now, for each \emph{ingoing} special edge of $v$, we
need to add a white elementary star with two special split-edges: the
multiplicative contribution for adding such a star is 
$\sum_{k\in D} [(m-1)k-1] {mk-1 \choose k}  \zc^{k-1}
\Tplanarc^{(m-1)k}=\frac{1}{(m-1)\zc}$. The multiplicative factor
for the second case is therefore:
\begin{eqnarray*}&&
\frac{(m-1)(m-2)}{2}\zc
\left[\frac{1}{(m-1)\zc}\right]^{3-out(v)} 
\end{eqnarray*}
Putting the two cases together, the multiplicative contribution of a vertex of
the type {\bf 3.2} is:
\begin{eqnarray*}&&
[\zc(m-1)]^{out(v)} 
\left(
\frac{1}{\zc^2\Tplanarc} \frac{\gamma_{m,D}-(m-2)\beta_{m,D}}{2(m-1)^2}
+\frac{m-2}{2(m-1)^2\zc^2}\right)\\
&=&
[\zc(m-1)]^{out(v)-2} \frac{\gamma_{m,D}}{2\Tplanarc}
\end{eqnarray*}
where we used that $\Tplanarc=\beta_{m,D}$.

\subsection{Final asymptotics}
Putting the four cases together, it finally comes that:
\begin{eqnarray*}&&
\sum_{(F,\mathfrak{a})\in D_{\mathfrak{s},\lambda,\tau}}
\zc^{(\mathfrak{s},\tau,F,\mathfrak{a},\lambda)}\\
&=&
\prod_{v:type \ 1}
\frac{\gamma_{m,D}}{2}
\prod_{v:type \ 2}
\frac{[\zc(m-1)]^{out(v)-1}}{\Tplanarc} \frac{\gamma_{m,D}}{2} \\
&& \ \ 
\prod_{v:type \ 3.1}
{[\zc(m-1)]^{out(v)-1}} \frac{\gamma_{m,D}}{2\Tplanarc^2}
\prod_{v:type \ 3.2}
[\zc(m-1)]^{out(v)-2} \frac{\gamma_{m,D}}{2\Tplanarc}\\
&=&
\left(\frac{\gamma_{m,D}}{2}\right)^{|V(\mathfrak{s})|}
[\zc(m-1)]^{out(\mathfrak{s})-v_2-v_3^{(1)}-2v_3^{(2)}}
\Tplanarc^{-v_2-2v_3^{(1)}-v_3^{(2)}}
\end{eqnarray*}
where $out(\mathfrak{s})=\sum_{v\ type\  2;3.1;3.2} out(v)$ is the total number
of special half-edges that are outgoing. Observe that $out(\mathfrak{s})$ is
also the total number of special edges (since each edge has exactly one outgoing half-edge),
i.e. $out(\mathfrak{s})=n_{\neq}$.
Moreover, since
$v_3^{(1)}=v_3^{(2)}$, we
have: $v_2+v_3^{(1)}+2v_3^{(2)}=v_2+\frac{3}{2}v_3=n_{\neq}$.

Hence the multiplicative factor corresponding to all decorations of
$\mathfrak{s},\tau,\lambda$ is: 
$$
\sum_{(F,\mathfrak{a})\in D_{\mathfrak{s},\lambda,\tau}}
\zc^\mathfrak{(\mathfrak{s},\tau,F,\mathfrak{a},\lambda)}=
\left(\frac{1}{\Tplanarc}\right)^{n_{\neq}}\left(\frac{\gamma_{m,D}}{2}\right)^{|V(\mathfrak{s})|}.
$$
That is where something great happens: the factor
$\left(\frac{1}{\Tplanarc}\right)^{n_{\neq}}$ simplifies with
$\Tplanarc^{n_{\neq}}$ in Equation~\ref{eq:devRfz}. Hence, the first term in
the singular expansion of $R_{\mathfrak{s},\lambda,\tau}$ does not depend on
the typing ! This is, with Lemma~\ref{lemma:dimension}, the main argument
leading to Theorem~\ref{thm:hypermaps}. Precisely, summing 
 Equation~\ref{eq:devRfz} over all the decorations gives:
\begin{eqnarray*}&&
R_{\mathfrak{s},\lambda,\tau}(z)\\&=&
c_{\mathfrak{s},\lambda}
\left(\frac{\gamma_{m,D}}{2}\right)^{|V(\mathfrak{s})|}
(m-1)^{\frac{k+m}{4}}m^{\frac{M-k}{2}}
\gamma_{m,D}^{\frac{M-3k}{4}}
\beta_{m,D}^{-\frac{k+M}{4}}
3^{\frac{k-M}{2}}
2^{\frac{-k-5M}{4}}
\left(1-\frac{z}{\zc}\right)^{-\frac{k+M}{4}}
[1+o(1)]
\end{eqnarray*}
where $k=6g-3$ and $|V(\mathfrak{s})|=M+1=4g-2$. This gives our main estimate:
\begin{proposition}
When $z$ tends to $\zc$, the following Puiseux expansion holds:
\begin{eqnarray}\label{eq:devRaftersum}
R_{\mathfrak{s},\lambda,\tau}(z)=
c_{\mathfrak{s},\lambda}
(m-1)^{\frac{5g-3}{2}}m^{-g}
\gamma_{m,D}^{\ \frac{g-1}{2}}
\beta_{m,D}^{\ \frac{3-5g}{2}}
3^{g}
2^{\frac{13-21g}{2}}
\left(1-\frac{z}{\zc}\right)^{-\frac{k+M}{4}}
[1+o(1)]
\end{eqnarray}
\end{proposition}

Let $d=\gcd(D)$. Then $\Tplanar(z)$ is actually a series in $z^d$. It has
therefore at least $d$ dominant singularities, which are the $\zc\xi^k$ for a
primitive $d$-th root of unity $\xi$. Now, the positivity of the
coefficients in Equation~\ref{eq:planar} easily shows that these are the only
singularities of $\Tplanar(z)$, and hence of $t(z)$. Hence, due to the
compositional nature of the series $R_{\mathfrak{f}}(z)$ (up to the prefactor
$z^\mathfrak{f}$, $R_{\mathfrak{f}}(z)$ is in fact a power series with
positive coefficients in $t(z)$), this implies that the $\zc\xi^k$ are the $d$
only dominant roots of $R_{\mathfrak{f}}(z)$ for all $\mathfrak{f}$, so that
they are also the $d$ only dominant roots of $R_{\mathfrak{s},\tau,\lambda}(z)$.

Now, $R_{\mathfrak{s},\tau,\lambda}(z)$ being an algebraic series, it is
amenable to singularity analysis, in the classical sense of \cite{FlOd}. Hence
Equation~\ref{eq:devRaftersum} and the classical transfert theorems of
\cite{FlOd} imply that the coefficient of $z^n$ in
$R_{\mathfrak{s},\tau,\lambda}(z)$ satisfies: $$
[z^n]R_{\mathfrak{s},\lambda,\tau}(z)\sim
\frac{dc_{\mathfrak{s},\lambda}}{\Gamma\left(\frac{5g-3}{2}\right)}
(m-1)^{\frac{5g-3}{2}}m^{-g}
\gamma_{m,D}^{\ \frac{g-1}{2}}
\beta_{m,D}^{\ \frac{3-5g}{2}}
3^{g}
2^{\frac{13-21g}{2}}\cdot
n^{\frac{5g-5}{2}} \zc^{-n}
$$
when $n$ goes to infinity along multiples of $d$.
Using Corollary~\ref{cor:algo} and Theorem~\ref{thm:bij}, we obtain that the
number $h^\bullet_{g,m,D}(n)$ of \emph{rooted and pointed} $m$-hypermaps of
degree set $mD$ with $n$ black faces satisfies, when $n$ tends to infinity
along multiples of $d$: $$
h^\bullet_{g,m,D}(n) \sim 
\frac{dc_g}{\Gamma\left(\frac{5g-3}{2}\right)}
(m-1)^{\frac{5g-3}{2}}m^{1-g}
\gamma_{m,D}^{\ \frac{g-1}{2}}
\beta_{m,D}^{\ \frac{3-5g}{2}}
3^{g}
2^{\frac{11-21g}{2}}\cdot
n^{\frac{5g-3}{2}} \zc^{-n}
$$
where $c_g=\frac{m^{2g}}{6g-3}
\sum_{(s,\lambda)\in\mathcal{P}_g}c_{\mathfrak{s},\lambda}$ ; observe the
factor $m^{2g}$, that comes from Lemma~\ref{lemma:dimension}.

Moreover, it follows from the remark after Corollary~\ref{cor:algo} that the
number $c^\bullet_{g,m,D}(n)$ of \emph{rooted and pointed} $m$-constellations
of degree set $mD$ with $n$ black faces satisfies, when $n$ tends to infinity
along multiples of $d$: $$
m^{2g} c^\bullet_{g,m,D}(n) \sim h^\bullet_{g,m,D}(n).
$$

\subsection{A ``de-pointing lemma''.}
The last thing that remains to do to prove Theorems~\ref{thm:main} and
~\ref{thm:hypermaps} is to relate maps which are both rooted and pointed to
maps which are only rooted. First, observe that each rooted map with $v$
vertices corresponds to exactly $v$ distinct rooted and pointed maps. Moreover,
the vertices of a $m$-hypermap correspond, except for the pointed vertex, to the
labelled vertices of its mobile.

Now, let $\mathfrak{t}_n$ be a mobile corresponding to a $m$-hypermap of degree
set $mD$ and size $n$, chosen uniformly at random. We let $n=n_1+n_2$ be the
number of black vertices of $\mathfrak{t}_n$, where $n_1$ is the number of
black vertices appearing in the superchains of $\mathfrak{t}_n$ or in the
decoration of its scheme, and $n_2$ is the number of black vertices appearing in
the ``planar parts'' that are attached on it.

First, the compositional nature of the series $R_\mathfrak{f}(z)$, which obeys a
composition schema of exponents $\left(\frac{3-5g}{2}\right)\circ\frac{1}{2}$
in the terminology of \cite{BaFlScSo}, implies that $n_1/n$ converges in
probability to $0$ (and actually more, namely that $n_1/\sqrt{n}$ converges to
a real random variable whose density is a slight modification of a Gaussian
law). Equivalently $n_2/n\rightarrow 1 $ in probability. For the same reason,
if $l_n$ denotes the total number of labelled vertices of $\mathfrak{t}_n$,
then $l'_n/l_n\rightarrow 1$ in probability, where $l'_n$ is the number of
labelled vertices present in the planar parts.

Now, conditionnally to $n_2$, those planar parts form a random forest of planar
mobiles, chosen uniformly at random among those with a total of $n_2$ black
vertices. The generating series of planar mobiles, where $z$ counts black
vertices and $u$ counts the number of labelled vertices is:
\begin{eqnarray}\label{eq:Tuz}
T(z,u) = u + \sum_{k\in D}{mk-1 \choose k }z^k T(z,u)^{(m-1)k}
\end{eqnarray}
Hence, according to the famous theorem of Drmota \cite{Drmota:RSA} concerning
the distribution of the numbers of terminal symbols in large words of context
free languages, the following convergence in probability holds:
$$
\frac{l'_n}{n_2} \rightarrow \frac{T'_u(\zc,1)}{\zc T'_z(\zc,1)}
$$
when $n_2\rightarrow \infty$.
Now, it easy to obtain from Equation~\ref{eq:Tuz} that $T'_u(\zc,1)=1$, so
that: $$
\frac{T'_u(\zc,1)}{\zc T'_z(\zc,1)} = 
\frac
{1}
{\sum_{k\in D} k{mk-1 \choose k} {\zc}^k {\Tplanarc}^{(m-1)k}}
=\frac{m-1}{\beta_{m,D}}.
$$
Hence we have: $\frac{l_n}{n}\rightarrow \frac{m-1}{\beta_{m,D}}$ in
probability. This gives:
\begin{lemma}
The numbers of rooted and pointed, and rooted only $m$-hypermaps or
$m$-constellations are related by the following asymptotic relations, when $n$
tends to infinity along multiples of $d$:
$$
h^\bullet_{g,m,D}(n) \sim \frac{(m-1)n}{\beta_{m,D}} 
h_{g,m,D}(n)
\ \ \ \ ; \ \ \ \
c^\bullet_{g,m,D}(n) \sim \frac{(m-1)n}{\beta_{m,D}} 
c_{g,m,D}(n).
$$
\end{lemma}
This last result  completes the proof of Theorems~\ref{thm:main} 
and~\ref{thm:hypermaps}, up to setting
$$
t_g = \frac{c_g3^g2^{7-11g}}{(6g-3)\Gamma\left(\frac{5g-3}{2}\right)}.
$$
The last thing to do is to check that $t_g$ is indeed the same constant as in
\cite{BeCa}: this will be done in the next and last subsection, where we examine
some corollaries of the two theorems.

\subsection{The case $D=\{k\}$.}

In this subsection, we examine the case $D=\{k\}$.
In this case, we have:
$$\displaystyle
[(m-1)k-1] {mk-1 \choose k} t_c^{k}=1
$$
which gives:
\begin{eqnarray*}
\beta_{m,k} &=& \frac{(m-1)k}{(m-1)k-1} \\
\gamma_{m,k} &=& (m-1) k
\end{eqnarray*}
We obtain the following:
\begin{corollary}
Let $m\geq 2$ and $k\geq 2$ be integers. Then the number $c_{g,m,k}(n)$ of
rooted $m$-constellations of genus $g$ and size $n$, and whose all white faces
have degree $mk$ satisfies, when $n$ tends to infinity along multiples of $k$:
\begin{eqnarray*}
c_{g,m,k}(n) \sim
t_g \frac{k}{2} \left(
\frac{\sqrt{2}\sqrt{m-1}[(m-1)k-1]^{\frac{5}{2}}}{m k^2}\right)^{g-1}
n^{\frac{5(g-1)}{2}} {(z_{m,k}^{(c)})}^{-n}
\end{eqnarray*}
where:
$\displaystyle\ \ 
z_{m,k}^{(c)} =
\left[\frac{(m-1)k}{(m-1)k-1}\right]^{1-m}\left[[(m-1)k-1]{mk-1
\choose k}\right]^{-\frac{1}{k}} .$
\end{corollary}
For $m=2$, we  obtain the asymptotic number of bipartite $2k$-angulations with
$n$ edges:
\begin{eqnarray*}
c_{g,2,k}(n) \sim
t_g \frac{k}{2} \left[\frac{1}{\sqrt{2}} \frac{(k-1)^{5/2}}{k^2}\right]^{g-1}
n^{\frac{5(g-1)}{2}} {z^{(c)}_{2,k}}^{-n}
\end{eqnarray*}
If furthermore $k=2$, we recover the asymptotic number of
bipartite quadrangulations with $2n$ edges (which is also the number of maps
with $n$ edges, thanks to the classical bijection of Tutte), 
in accordance with \cite{BeCa0,ChMaSc}):
\begin{corollary}
The number $m_n^{(g)}$ of rooted maps on $\mathcal{S}_g$ with $n$ edges
satisfies: $$
m_n^{(g)} \sim t_g n^{\frac{5(g-1)}{2}} 12^n
$$
\end{corollary}
In particular, this proves that our constant $t_g$ is indeed the same as the one
introduced in \cite{BeCa0}.
Our last corollary concerns the number of all $m$-constellations of genus $g$
(without degree restriction).
The following lemma is classical and reduces the study of all $m$-constellations (without degree restriction) to the
study of degree restricted $m+1$-constellations. See the proof of Corollary 2.4 in \cite{MBM-Sc}.
\begin{lemma}
There is a bijection between rooted $m$-constellations with $n$ black faces and 
rooted $m+1$-constellations  with $n$ black faces where all white faces have
degree $m+1$.
\end{lemma}
This implies
\begin{corollary}
The number of all rooted $m$-constellations with $n$ black faces on a surface of
genus $g$ is asymptotically equivalent to:
\begin{eqnarray*}
\frac{t_g}{2} \left(
\frac{\sqrt{2m}(m-1)^{5/2}}{m+1}\right)^{g-1}
n^{\frac{5(g-1)}{2}} \left(\frac{m^{m+1}}{(m-1)^{m-1}}\right)^{n}.
\end{eqnarray*}
\end{corollary}
\vspace{0mm}
\subsubsection*{\bf Acknowledgements.}
The author thanks Gilles Schaeffer for his help and support. Thanks also to
Mireille Bousquet-M\'elou for stimulating discussions.

%\nocite{*}
\bibliographystyle{alpha}
% use the following instead if you encounter problems 
%\bibliographystyle{alpha}
\bibliography{Ch07}
\label{sec:biblio}
	
\end{document}